\documentclass[11pt]{amsart}

\usepackage[colorlinks = true,
            linkcolor = red,
            urlcolor  = magenta,
            citecolor = black,
            anchorcolor = blue]{hyperref}
\usepackage{color}
\usepackage[shortalphabetic,initials,nobysame]{amsrefs}
\usepackage{upgreek}
\usepackage{tikz}
\usepackage[applemac]{inputenc} % for Macs
\usetikzlibrary{arrows}
\usepackage{xcolor}
\usepackage[all]{xy}
\usepackage{graphicx}
\usepackage{wrapfig}
\usepackage{yfonts}
\usepackage{graphicx}
\usepackage{amssymb}
\usepackage{epstopdf}
\usepackage{mathrsfs}
\usepackage[mathcal]{eucal}
\usepackage{bm}

\renewcommand{\eprint}[1]{\href{https://arxiv.org/abs/#1}{#1}}

\DeclareMathOperator{\GL}{\mathrm{GL}}

\DeclareMathOperator{\Hom}{Hom}

\DeclareMathOperator{\Lie}{Lie}

\newcommand{\fg}{\frak{g}}

\BibSpec{article}{%
    +{}  {\PrintAuthors}                {author}
    +{,} { \textit}                     {title}
     +{,} {}                            {note}
    +{.} { }                            {part}
    +{:} { \textit}                     {subtitle}
    +{,} { \PrintContributions}         {contribution}
    +{.} { \PrintPartials}              {partial}
    +{,} { }                            {journal}
    +{}  { \textbf}                     {volume}
    +{}  { \PrintDate}                {date}
    +{,} { \issuetext}                  {number}
    +{,} { \eprintpages}                {pages}
    +{,} { }                            {status}
    +{,} { \DOI}                   {doi}
    +{,} { \eprint}        {eprint}
      +{,} {\publisher}                {publisher}
    +{,} { \address}                   {address}
    +{}  { \parenthesize}               {language}
    +{}  { \PrintTranslation}           {translation}
    +{;} { \PrintReprint}               {reprint}
    +{.} {}                             {transition}
    +{}  {\SentenceSpace \PrintReviews} {review}
}

\newtheorem{Thm}{Theorem}[section]
\newtheorem{Lem}[Thm]{Lemma}
\newtheorem{Prop}[Thm]{Proposition}
\newtheorem{Cor}[Thm]{Corollary}

\theoremstyle{definition}
\newtheorem{Def}[Thm]{Definition}
\theoremstyle{remark}
\newtheorem{Rem}[Thm]{Remark}

\newtheoremstyle{named}{}{}{\itshape}{}{\bfseries}{.}{.5em}{#1 #3}
\theoremstyle{named}

\def\C{\mathbb{C}}
\def\Z{\mathbb{Z}}

\def\P{\mathbb{P}}

\def\fb{\mathfrak{b}}
\def\g{\mathfrak{g}}
\def\Frenkel:2013uda{\mathfrak{h}}

\def\cF{\mathcal{F}}

\def\cL{\mathcal{L}}

\def\cO{\mathcal{O}}

\def\cV{\mathcal{V}}
\def\cW{\mathcal{W}}

\def\bo{\textbf{o}}

\def\=>{\Longrightarrow}

\def\to{\longrightarrow}

\def\o+{\oplus}
\def\bo+{\bigoplus}

\def\<{\langle}
\def\>{\rangle}
\def\({\left(}
\def\){\right)}

\def\^{\wedge}
\def\+{\dagger}

\def\dd[#1,#2]{\frac{d#1}{d#2}}
\def\del[#1,#2]{\frac{\partial #1}{\partial #2}}
\def\over[#1]{\overline{#1}}
\def\vec[#1]{\overrightarrow{#1}}

\makeatletter
\def\mr@ignsp#1 {\ifx\:#1\@empty\else #1\expandafter\mr@ignsp\fi}%
\newcommand{\multiref}[1]{\begingroup%\let\protect\string%
\xdef\mr@no@sparg{\expandafter\mr@ignsp#1 \: }%
\def\mr@comma{}%
\@for\mr@refs:=\mr@no@sparg\do{\mr@comma\def\mr@comma{,}\ref{\mr@refs}}%
\endgroup}
\makeatother

\newcommand{\hypref}[2]{\ifx\href\asklFrenkel:2013udaas #2\else\href{#1}{#2}\fi}
\newcommand{\Secref}[1]{Section~\multiref{#1}}

\tikzset{->-/.style={decoration={
  markings,
  mark=at position .5 with {\arrow{latex}}},postaction={decorate}}}
\tikzset{
    >=latex
    }

\newcommand{\wt}{\widetilde}

\newcommand{\nc}{\newcommand}
\nc{\on}{\operatorname}
\nc{\la}{\lambda}
\nc{\wh}{\widehat}
\nc{\ghat}{\wh\g}
\nc{\mb}{\mathbf}

\begin{document}
\title[$q$-Opers, $QQ$-systems, and Bethe ansatz II:
Generalized Minors]{$q$-Opers, $QQ$-systems, and Bethe ansatz II:\\
Generalized Minors}

\author[P. Koroteev]{Peter Koroteev}
\address{
Department of Mathematics,
University of California,
Berkeley, CA 94720, USA
and 
Rutgers University, Piscataway, NJ, 08854, USA\newline
 Email: \href{mailto:pkoroteev@math.berkeley.edu}{pkoroteev@math.berkeley.edu}\newline 
 \href{https://math.berkeley.edu/~pkoroteev/}{https://math.berkeley.edu/$\sim$pkoroteev}
}

\author[A.M. Zeitlin]{Anton M. Zeitlin}
\address{
          Department of Mathematics, 
          Louisiana State University, 
          Baton Rouge, LA 70803, USA\newline
Email: \href{mailto:zeitlin@lsu.edu}{zeitlin@lsu.edu},\newline
 \href{http://math.lsu.edu/~zeitlin}{http://math.lsu.edu/$\sim$zeitlin}
}

\date{\today}

\numberwithin{equation}{section}

\begin{abstract}

In this paper, we describe a certain kind of $q$-connections on a projective line, namely $Z$-twisted $(G,q)$-opers with regular singularities using the language of generalized minors. In part one \cite{Frenkel:2020} we explored the correspondence between these $q$-connections and $QQ$-systems/Bethe Ansatz equations. Here we associate to a $Z$-twisted $(G,q)$-oper a class of meromorphic sections of a $G$-bundle, satisfying certain difference equations, which we refer to as $(G,q)$-Wronskians. Among other things, we show that the $QQ$-systems and their extensions emerge as the relations between generalized minors, thereby putting the Bethe Ansatz equations in the framework of cluster mutations known in the theory of double Bruhat cells. 
\end{abstract}

\maketitle

\setcounter{tocdepth}{1}
\tableofcontents

\section{Introduction}
The concept of opers, specific connections on principal $G$-bundles ($G$ is a simple complex Lie group), was introduced by Beilinson and Drinfeld \cite{Beilinson:2005} in the context of the geometric Langlands correspondence. 
The idea of an oper goes back to the work of Drinfeld and Sokolov, where the local formulas for such connections emerged as Lax operators for KdV systems in the context of the celebrated Drinfeld-Sokolov reduction \cite{Drinfeld:1985}.

An important example of Geometric Langlands correspondence \cite{frenkel2007langlands} discovered and extensively studied by E. Frenkel and his collaborators \cite{Feigin:1994in,Frenkel:aa,Frenkel:ab} is between $G$-opers on the projective line with regular singularities and trivial monodromy and the Bethe equations, describing the spectrum of the Gaudin model for the Lie algebra of the Langlands dual group  $^LG$.

Recently a $q$-deformation of this example has been discovered, first, for $SL(N)$ in \cite{KSZ} and then in the case of simple simply connected complex $G$ in \cite{Frenkel:2020}, following the local consideration related to the $q$-analogue of the Drinfeld-Sokolov reduction \cite{1998CMaPh.192..631S} and related $q$-difference operators \cite{Mukhin_2005}. We refer to the analog of an oper connection as $(G,q)$-oper which has the structure of the $q$-connection introduced by Baranovsky and Ginzburg \cite{Baranovsky:1996}. 

On the other side of the correspondence, the Gaudin model was replaced by the Bethe ansatz equations for XXZ models, the integrable models based on quantum groups. A fascinating intermediate object needed for establishing the correspondence is the so-called $QQ$-system, a functional equation that produces the Bethe equations upon certain nondegeneracy conditions. The $QQ$-system previously emerged in the study of representations of quantum affine algebras \cite{Bazhanov:1998dq,Frenkel:2016} as well as the ODE/IM correspondence \cite{MRV1,MRV2}.

It turned out, that the concept of a $(G,q)$-oper and its treatment of the $QQ$-systems and its extensions appeared to be an effective tool in the study of quantum integrable systems such as in quantum/classical duality \cite{KSZ}, bispectral duality and even in enumerative geometry \cite{Koroteev:2017aa,Koroteev:2020tv,Koroteev:2021a} and physics \cite{Gaiotto:2013bwa,Bullimore:2015fr}. There also exists a deformation of the geometric Langlands correspondence in the example above, known as the quantum $q$-Langlands \cite{Aganagic:2017smx}.

Let us be more concrete about the explicit structure of $(G,q)$-opers as introduced in \cite{Frenkel:2020}. One can locally think about the $q$-oper connection locally as an element $A(z)\in G(z)\equiv G(\mathbb{C}(z))$. 
The so-called Miura condition implies that this $q$-connection is preserved by a certain reduction to the Borel 
subgroup. The zero monodromy condition is replaced in this case by the {\it $Z$-twisted} condition $A(z)=g(qz)Zg^{-1}(z)$, where $g(z)\in G(z)$, and $Z\in H=B/[B,B]\subset G$ is an element of the Cartan subgroup referred to as the twist parameter. For a regular $Z$ there is exactly $|W_G|$ (here $W_G$ is a Weyl group of $G$)  $Z$-twisted $(G,q)$-Miura opers associated to a given $(G,q)$-oper. 
To establish correspondence with the $QQ$-systems and as a byproduct, the Bethe ansatz equations, one introduces a milder condition by requiring this condition to hold on $(G,q)$-oper-invariant two-dimensional subspaces within fundamental highest weight (lowest weight in this manuscript) representations. In this case, one applies $Z$-twisted condition to $2\times 2$ matrices $A_i(z)$ ($GL(2,q)$-opers), one for every fundamental weight, thus producing a system of functional equations known as the $QQ$-system. In \cite{Frenkel:2020} we referred to a $q$-connection with this milder $Z$-twisted condition as Miura-Pl\"ucker $(G,q)$-oper.
We also introduced quantum B\"acklund transformations, following \cite{Mukhin_2005} for Miura-Pl\"ucker $(G,q)$-opers ($q$-gauge transformations of a certain kind), which produces another Miura-Pl\"ucker $(G,q)$-oper from a given one corresponding to $w(Z)$-twist, $w\in W_G$, provided that certain nondegeneracy conditions on the $QQ$-system hold. This effectively produced a new $QQ$-system for every $w\in W_G$, generating altogether what we call the {\it full $QQ$-system} (see \cite{Ekhammar:2020} for related studies on the full $QQ$-system). If the nondegeneracy conditions repeated themselves for each simple reflection until we reach the maximal element $w_0$ then, as we have shown in \cite{Frenkel:2020}, the corresponding Miura-Pl\"ucker $q$-oper becomes $Z$-twisted Miura $q$-oper. 

At the same time, we established in \cite{Koroteev:2020tv}, that in the case of $G=SL(N)$ these nondegeneracy conditions on the parts of the full $QQ$-system are actually redundant. We obtained that by following another approach to $(SL(N), q)$-opers, which involves the associated bundle produced by the defining representation. That approach heavily relies on the fact that the defining representation is exact.  The elements of the full $QQ$-system are produced by certain minors of the $q$-deformation of the Wronskian matrix, which is constructed from the components of a certain section of the line bundle (this section produces locally the entire fiber of the associated bundle via the action of the $(SL(N), q)$-oper connection \cite{KSZ}). The functional relations of the $QQ$-system emerge from certain determinant identities known since the 19th century between the corresponding minors. This led us to conjecture in \cite{Koroteev:2020tv} that for a general simple simply connected $G$, the $QQ$-system is produced from the relations between generalized minors, as studied in the works of Berenstein, Fomin, and Zelevinsky \cite{Berenstein_1997,Fomin:wy,FZ,Berenstein_2005}. However, the $(G,q)$-Wronskian matrix construction for $G$ was not clear, since one does not have the luxury of $SL(N)$ case. The cluster interpretation of Bethe equations (using the TQ-relation) was given previously by Hernandez and Leclerc in \cite{Hernandez:2016}.

This is exactly what we devote this paper to. Namely, we define the {\it  $(G,q)$-Wronskians} as meromorphic sections of a $G$-bundle by introducing the system of $q$-difference equations for the associated bundles corresponding to fundamental weights. We establish the correspondence between classes of $(G,q)$-Wronskians and $Z$-twisted Miura-Pl\"ucker opers, so that the elements of the $QQ$-system emerge as generalized minors (\{$\Delta_{w\cdot\omega_i, \omega_i}$\}, $w \in W_G$ in the notation of \cite{FZ}) of $(G,q)$-Wronskians. This reformulation, in particular, leads to the equivalence between Miura-Pl\"ucker $(G,q)$-opers and $Z$-twisted Miura $(G,q)$-opers as it was in the case of $SL(N)$. Finally, for each class of $(G,q)$-Wronskian we introduce a universal $(G,q)$-Wronskian, which in the case of $SL(N)$ coincides with the $q$-Wronskian matrix, modulo multiplication by a diagonal matrix.

\vskip.1in

\noindent\textbf{Future Directions.}
We hope that our results can be useful in a deeper understanding of the ODE/IM correspondence \cite{Bazhanov:1994ft,Bazhanov:1996dr,Bazhanov:1998dq,Dorey:2007zx} (see \cite{FF:kdv} for related ideas). 
In the nutshell, the ODE/IM relates vacuum eigenvalues of Baxter operators of the quantum KdV model associated with affine Lie algebra $\widehat{\mathfrak{g}}$ and spectral determinants of certain singular differential operators associated with affine opers associated with the Langlands dual $^L\widehat{\mathfrak{g}}$. 
It was shown in \cite{MRV1,MRV2} that for the standard quantum KdV system, the spectral determinants of the corresponding Strum-Liouville differential operators can be identified with solutions of the $QQ$-system (albeit with different analyticity conditions on the entire Q-functions).

As was already pointed out in \cite{Frenkel:2020}, our results anticipate an intriguing relationship between $(G,q)$-opers which we are discussing here, and the differential operators from the ODE/IM. Recall that according to \cite{MRV1}, Section 4, for every fundamental highest weight representation of ${\mathfrak{g}}$ there exists a pair of normalized (singular at the origin) solutions of the linear ODE $\chi^{(i)}(x), \varphi^{(i)}(x)$ which are the eigenvectors of the associated monodromy matrix to the ODE in question, which are associated (in a nontrivial way) to the top two weights in the representation. The main point of interest in ODE/IM correspondence is the study of the so-called $\Psi^{(i)}$-solutions for each such representation, with a certain prescribed asymptotic behavior at infinity. It turns out that each 
$\Psi^{(i)}$ is a linear combination of $\chi^{(i)}$ and $\varphi^{(i)}$ with coefficients being generators of the $QQ$-system (in this incarnation they are referred to as spectral determinants) and some terms corresponding to lower weights.
That allows us to speculate that one can relate 
the q-connection $A(z)$ of the $(G,q)$-opers to the monodromy matrix of the ODE from \cite{MRV1}. Indeed, in our construction of the Miura-Pl\"ucker $q$-oper (see formula (\ref{odeim}) and the Theorem \ref{1to1wrop}) the decomposition of  $\Psi^{(i)}$ solution in terms of $\chi^{(i)}(x), \varphi^{(i)}(x)$ corresponds to 
exactly the first terms in the decomposition of the action of a $(G,q)$-Wronskian (associated to a Miura $(G,q)$-oper) on the highest weight vector. Thus there seems to be an interesting duality between $(G,q)$-opers and affine opers for $^L\widehat{\mathfrak{g}}$ -- the former objects are related to q-difference operators, the latter to differential operators. We plan to pursue this idea in future publications.

\vskip.1in

\noindent\textbf{Structure of the Paper.}
This paper is organized as follows. In \Secref{regsing} we set up the group-theoretic data which will be needed for our construction. In \Secref{Sec:MiuraQQ} we review basic definitions and theorems of \cite{Frenkel:2020} using the new conventions. The final \Secref{Sec:GeneralizedWronskians} contains new results about generalized minors and quantum Wronskians as well as an illustration of these results for $G=SL(r+1)$.

\vskip.1in
\noindent\textbf{Acknowledgements.}
P.K. is partially supported by AMS Simons Travel Grant. A.M.Z. is partially supported by Simons Collaboration Grant, Award ID: 578501, and NSF grant DMS-2203823. We are grateful to Edward Frenkel for his support and valuable comments.

\section{Group-theoretic data}    \label{regsing} 
Let $G$ be a connected, simply connected, simple algebraic group of
rank $r$ over $\mathbb{C}$.  We fix a Borel subgroup $B_-$ with
unipotent radical $N_-=[B_-,B_-]$ and a maximal torus $H\subset B_-$.
Let $B_+$ be the opposite Borel subgroup containing $H$.  Let $\{
\alpha_1,\dots,\alpha_r \}$ be the set of positive simple roots for
the pair $H\subset B_+$.  Let $\{ \check\alpha_1,\dots,\check\alpha_r
\}$ be the corresponding coroots; the elements of the Cartan matrix of
the Lie algebra $\fg$ of $G$ are given by $a_{ij}=\langle
\alpha_j,\check{\alpha}_i\rangle$. The Lie algebra $\fg$ has Chevalley
generators $\{e_i, f_i, \check{\alpha}_i\}_{i=1, \dots, r}$, so that
$\fb_-=\Lie(B_-)$ is generated by the $f_i$'s and the
$\check{\alpha}_i$'s and $\fb_+=\Lie(B_+)$ is generated by the $e_i$'s
and the $\check{\alpha}_i$'s.  Let $\omega_1,\dots\omega_r$ be the
fundamental weights, defined by $\langle \omega_i,
\check{\alpha}_j\rangle=\delta_{ij}$.

Let $W_G=N(H)/H$ be the Weyl group of $G$. Let $w_i\in W$, $(i=1,
\dots, r)$ denote the simple reflection corresponding to
$\alpha_i$. We also denote by $w_0$ be the longest element of $W$, so
that $B_+=w_0(B_-)$.  Recall that a Coxeter element of $W$ is a
product of all simple reflections in a particular order. It is known
that the set of all Coxeter elements forms a single conjugacy class in
$W_G$. We will fix once and for all (unless otherwise specified) a
particular ordering $(\alpha_{i_1},\ldots,\alpha_{i_r})$ of the simple
roots. Let $c=w_{i_1}\dots w_{i_r}$ be the Coxeter element associated
to this ordering. In what follows (unless otherwise specified), all
products over $i \in \{ 1, \dots, r \}$ will be taken in this order;
thus, for example, we write $c=\prod_i w_i$.  We also fix
representatives $s_i\in N(H)$ of $w_i$. In particular, $s=\prod_i s_i$
will be a representative of $c$ in $N(H)$.

Although we have defined the Coxeter element $c$ using $H$ and $B_-$,
it is in fact the case that the Bruhat cell $BcB$ makes sense for any
Borel subgroup $B$.  Indeed, let $(\Phi,\Delta)$ be the root system
associated to $G$, where $\Delta$ is the set of simple roots as
above and $\Phi$ is the set of all roots.  These data give a
realization of the Weyl group of $G$ as a Coxeter group, i.e., a pair
$(W_G,S)$, where $S$ is the set of Coxeter generators $w_i$ of $W_G$
associated to elements of $\Delta$. Now, given any Borel subgroup $B$,
set $\fb = \on{Lie}(B)$. Then the dual of the vector space
$\fb/[\fb,\fb]$ comes equipped with a set of roots and
simple roots, and this pair is canonically isomorphic to the root
system $(\Phi,\Delta)$~\cite[\S 3.1.22]{CG}.  The definition of the
sets of roots and simple roots on this space involves a choice of
maximal torus $T\subset B$, but these sets turn out to be independent
of the choice. Accordingly, the group $N(T)/T$ together with the set
of its Coxeter generators corresponding to these simple roots is
isomorphic to $(W_G,S)$ as a Coxeter group. Under this isomorphism,
$w\in W_G$ corresponds to an element of $N(T)/T$ by the following
rule: we write $w$ as a word in the Coxeter generators of $W_G$
corresponding to elements of $S$ and then replace each Coxeter
generator in it by the corresponding Coxeter generator of
$N(T)/T$. Accordingly, the Bruhat cell $BwB$ is well-defined for any
$w\in W_G$.

\section{Miura $(G,q)$-opers with regular singularities and $QQ$-systems}\label{Sec:MiuraQQ}

In this section, we reformulate basic definitions and results from \cite{Frenkel:2020} using the notation and conventions which needed in this paper.

\subsection{$q$-connections and the structure of $(G,q)$-opers}
Given a principal $G$-bundle $\cF_G$ over $\P^1$ (in the Zariski
topology), let $\cF_G^q$ denote its pullback under the map $M_q: \P^1
\to \P^1$ sending $z\mapsto qz$. A meromorphic $(G,q)$-{\em
  connection} on a principal $G$-bundle $\cF_G$ on $\P^1$ is a section
$A$ of $\Hom_{\cO_{U}}(\cF_G,\cF_G^q)$, where $U$ is a certain Zariski open
dense subset of $\P^1$. We can always choose $U$ so that the
restriction $\cF_G|_U$ of $\cF_G$ to $U$ is isomorphic to the trivial
$G$-bundle. Choosing such an isomorphism, i.e. a trivialization of
$\cF_G|_U$, we also obtain a trivialization of
$\cF_G|_{M_q^{-1}(U)}$. Using these trivializations, the restriction
of $A$ to the Zariski open dense subset $U \cap M_q^{-1}(U)$ can be
written as a section of the trivial $G$-bundle on $U \cap M_q^{-1}(U)$,
and hence as an element $A(z)$ of $G(z)$.
Changing the trivialization of $\cF_G|_U$ via $g(z) \in G(z)$ changes
$A(z)$ by the following $q$-{\em gauge transformation}:
\begin{equation}    \label{gauge tr}
A(z)\mapsto g(qz)A(z)g(z)^{-1}.
\end{equation}
This shows that the set of equivalence classes of pairs $(\cF_G,A)$ as
above is in bijection with the quotient of $G(z)$ by the $q$-gauge
transformations \eqref{gauge tr}.

We define a $(G,q)$-oper in the following way (see \cite{Frenkel:2020})\footnote{In this paper we use slightly different conventions than in \cite{Frenkel:2020}, namely  we intechange $B_+$ and $B_-$}

\begin{Def}    \label{qop}
  A meromorphic $(G,q)$-{\em oper} (or simply a $q$-{\em oper}) on
  $\mathbb{P}^1$ is a triple $(\cF_G,A,\cF_{B_+})$, where $A$ is a
  meromorphic $(G,q)$-connection on a $G$-bundle $\cF_G$ on
  $\mathbb{P}^1$ and $\mathcal{F}_{B_+}$ is a reduction of $\cF_G$
  to $B_+$ satisfying the following condition: there exists a Zariski
  open dense subset $U \subset \P^1$ together with a trivialization
  $\imath_{B_+}$ of $\mathcal{F}_{B_+}$ such that the restriction of
  the connection $A: \cF_G \to \cF_G^q$ to $U \cap M_q^{-1}(U)$,
  written as an element of $G(z)$ using the trivializations of
  ${\mathcal F}_G$ and $\cF_G^q$ on $U \cap M_q^{-1}(U)$ induced by
  $\imath_{B_+}$, takes values in the Bruhat cell $B_+(\C[U \cap
  M_q^{-1}(U)]) c B_+(\C[U \cap M_q^{-1}(U)])$.
\end{Def}

Since $G$ is assumed to be simply connected, any $q$-oper connection
$A$ can be written (using a particular trivialization $\imath_{B_+}$)
in the form
\begin{equation}    \label{qop1}
A(z)=n'(z)\prod_i (\phi_i(z)^{-\check{\alpha}_i} \, s_i )n(z),
\end{equation}
where $\phi_i(z) \in\C(z)$ and  $n(z), n'(z)\in N_+(z)$ are such that
their zeros and poles are outside the subset $U \cap M_q^{-1}(U)$ of
$\P^1$.

Notice that in the above characterization one can view $(G,q)$-oper locally on $U \cap M_q^{-1}(U)$ as 
an equivalence class of the $q$-connections of the form \eqref{qop1} under the action of $q$-gauge transformations in $N_+(z)$, since the Cartan action is $q$-gauge fixed by the choice of $\{\phi_i(z)\}^r_{i=1}$. One can describe the set of representatives of these equivalence classes of $q$-connections as follows (see \cite{1998CMaPh.192..631S} and \cite{Frenkel:2020}).

\begin{Thm}\label{canform}
There exist a unique element $u(z) \in N_+(z)$ such that 
\begin{eqnarray}
u(qz)A(z)u^{-1}(z)\in N^s_+(z)\prod_i(\phi_i(z)^{-\check{\alpha}_i} \, s_i ),
\end{eqnarray} 
where $N_+^s=N_+\cap sN_-s^{-1}$.
\end{Thm}

\subsection{Miura $(G,q)$-opers}

The Miura $(G,q)$-opers are $q$-opers together
with an additional datum: a reduction of the underlying $G$-bundle to
the Borel subgroup $B_-$ (opposite to $B_+$) that is preserved by the
oper $q$-connection.

\begin{Def}    \label{Miura}
  A {\em Miura $(G,q)$-oper} on $\mathbb{P}^1$ is a quadruple
  $(\cF_G,A,\cF_{B_+},\cF_{B_-})$, where $(\cF_G,A,\cF_{B_+})$ is a
  meromorphic $(G,q)$-oper on $\P^1$ and $\cF_{B_-}$ is a reduction of
  the $G$-bundle $\cF_G$ to $B_-$ that is preserved by the
  $q$-connection $A$.
\end{Def}

Forgetting $\cF_{B_-}$, we associate a $(G,q)$-oper to a given Miura
$(G,q)$-oper. We will refer to it as the $(G,q)$-oper underlying the
Miura $(G,q)$-oper.

Suppose we are given a principal $G$-bundle $\cF_G$ on any smooth
complex manifold $X$ equipped with reductions $\cF_{B_+}$ and
$\cF_{B_-}$ to $B_+$ and $B_-$ respectively. We then assign to any
point $x \in X$ an element of the Weyl group $W_G$. To see this, first,
note that the fiber
$\cF_{G,x}$ of $\cF_G$ at $x$ is a $G$-torsor with reductions
$\cF_{B_+,x}$ and $\cF_{B_-,x}$ to $B_+$ and $B_-$
respectively. Choose any trivialization of $\cF_{G,x}$, i.e. an
isomorphism of $G$-torsors $\cF_{G,x} \simeq G$. Under this
isomorphism, $\cF_{B_+,x}$ gets identified with $aB_+ \subset G$ and
$\cF_{B_-,x}$ with $bB_-$. Then, $a^{-1}b$ is a well-defined element of
the double quotient $B_+\backslash G/B_-$, which is in bijection with
$W_G$. Hence, we obtain a well-defined element of $W_G$.

We will say that $\cF_{B_+}$ and $\cF_{B_-}$ have  {\em generic
  relative position} at $x \in X$ if the element of $W_G$ assigned to
them at $x$ is equal to $1$. This means that the corresponding element
$a^{-1}b$ belongs to the open dense Bruhat cell $B_+B_- \subset
G$.

\begin{Thm} \cite{Frenkel:2020}  \label{gen rel pos}
  For any Miura $(G,q)$-oper on $\mathbb{P}^1$, there exists an open
  dense subset $V \subset \P^1$ such that the reductions $\cF_{B_+}$
  and $\cF_{B_-}$ are in generic relative position for all $x \in V$.
\end{Thm}

Let $U$ be a Zariski open dense subset on $\P^1$ as in Definition
  \ref{qop}. Choosing a trivialization $\imath_{B_+}$ of $\cF_G$ on $U
  \cap M_q^{-1}(U)$, we can write the $q$-connection $A$ in the form
  \eqref{qop1}. On the other hand, using the $B_-$-reduction
  $\cF_{B_-}$, we can choose another trivialization of $\cF_G$ on $U
  \cap M_q^{-1}(U)$ such that the $q$-connection $A$ acquires the form
  $\wt{A}(z) \in B_-(z)$. Hence the following Corollary holds, which is a local description of the Theorem above.
  
\begin{Cor}\cite{Frenkel:2020} 
There exists $g(z) \in B_+(z) N_-(z)\subset G(z)$ such that
\begin{equation}    \label{connecting}
g(zq) n'(z)\prod_i (\phi_i(z)^{-\check{\alpha}_i} \, s_i )n(z)
g(z)^{-1} = \wt{A}(z) \in B_+(z).
\end{equation}
\end{Cor}

As an immediate consequence, we obtain the following Proposition.

\begin{Prop}   \cite{Frenkel:2020}  \label{gen rel pos1}
For any Miura $(G,q)$-oper on $\mathbb{P}^1$, there exists a
trivialization of the underlying $G$-bundle $\cF_G$ on an open
dense subset of $\P^1$ for which the oper $q$-connection has the form
\begin{equation}    \label{genmiura}
A(z)\in N_+(z)\prod_i([\phi_i(z)]^{-\check{\alpha}_i}s_i
)N_+(z) \; \cap \; B_-(z).
\end{equation}
\end{Prop}

The following Theorem allows us to describe the explicit form of the Miura $(G,q)$-opers:

\begin{Thm}  \cite{Frenkel:2020}  \label{gen elt}
Every element of the set
$N_+(z)\prod_i([\phi_i(z)]^{-\check{\alpha}_i}s_i)N_+(z) \; \cap \; B_-(z)$ can be written
in the form
\begin{equation}    \label{gicheck}
\prod_i \Big[g_i(z)\Big]^{-\check{\alpha}_i}e^{\frac{\phi_i(z) t_i}{g_i(z)}f_i}, \qquad
g_i(z)\in \mathbb{C}(z)^{\times},
\end{equation}
where each $t_i$ is determined by the lifting $s_i$.
\end{Thm}

\subsection{Miura $(G,q)$-opers with regular singularities}

It is clear that $(G,q)$-opers depend on the choice of the lift of the Coxeter element $c$ to $G(z)$, namely the functions $\{\phi_i(z)\}_{i=1,\dots, r}$. We will be interested in the case when these are 
polynomial functions.

Let $\{ \Lambda_i(z) \}_{i=1,\ldots,r}$ be a collection of
nonconstant polynomials.

\begin{Def}  \cite{Frenkel:2020}   \label{d:regsing}\label{MiuraRS}
  A $(G,q)$-{\em oper with regular singularities determined by $\{
    \Lambda_i(z) \}_{i=1,\ldots,r}$} is a $q$-oper on $\P^1$ whose
  $q$-connection \eqref{qop1} may be written in the form
\begin{equation}    \label{Lambda}
A(z)= n'(z)\prod_i([\Lambda_i(z)]^{-\check{\alpha}_i} \, s_i)n(z), \qquad
n(z), n'(z)\in N_+(z).
\end{equation}
  {\em A Miura $(G,q)$-oper with regular singularities determined by
polynomials $\{ \Lambda_i(z) \}_{i=1,\ldots,r}$} is a Miura
  $(G,q)$-oper such that the underlying $q$-oper has
regular singularities determined by $\{ \Lambda_i(z)
\}_{i=1,\ldots,r}$.
\end{Def}

Recall Theorem \ref{gen elt}. Observe that we can choose liftings
$s_i$ of the simple reflections $w_i \in W_G$ in such a way that
$t_i=1$ for all $i=1,\ldots,r$. From now on, we will only consider
such liftings.

Theorem \ref{gen elt} leads to the following statement: 

\begin{Thm} \cite{Frenkel:2020}    \label{gen elt1}
For every Miura $(G,q)$-oper with regular singularities determined by
the polynomials $\{ \Lambda_i(z) \}_{i=1,\ldots,r}$, the underlying
$q$-connection can be written in the form \eqref{form of A}:
\begin{equation}    \label{form of A}
A(z)=\prod_i
[g_i(z)]^{-\check{\alpha}_i} \; e^{\frac{\Lambda_i(z)}{g_i(z)}f_i}, \qquad
g_i(z) \in \C(z)^\times.
\end{equation}
\end{Thm}

\subsection{$Z$-twisted $(G,q)$-opers} Next, we consider a class of (Miura) $q$-opers that are gauge
equivalent to a constant element of $G$ (as $(G,q)$-connections).  Let
$Z$ be an element of the maximal torus $H$. Since $G$ is
simply connected, we can write
\begin{equation}    \label{Z}
Z = \prod_{i=1}^r \zeta_i^{-\check\alpha_i}, \qquad \zeta_i \in
\C^\times.
\end{equation}
\begin{Def} \cite{Frenkel:2020}    \label{Ztwoper}
1)  A {\em $Z$-twisted $(G,q)$-oper} on $\mathbb{P}^1$ is a $(G,q)$-oper
  that is equivalent to the constant element $Z \in H \subset H(z)$
  under the $q$-gauge action of $G(z)$, i.e. if $A(z)$ is the
  meromorphic oper $q$-connection (with respect to a particular
  trivialization of the underlying bundle), there exists $g(z) \in
  G(z)$ such that
\begin{eqnarray}    \label{Ag}
A(z)=g(qz)Z g(z)^{-1}.
\end{eqnarray}

2) A {\em $Z$-twisted Miura $(G,q)$-oper} is a Miura $(G,q)$-oper on
$\mathbb{P}^1$ that is equivalent to the constant element $Z \in H
\subset H(z)$ under the $q$-gauge action of $B_+(z)$, i.e. $g(z)\in B_+(z)$ in (\ref{Ag}).
\end{Def}

The following Proposition relates $Z$-twisted$(G,q)$-opers and their Miura counterparts.

\begin{Prop}    \label{Z prime}
1)  Let $Z \in H$. For any $Z$-twisted $(G,q)$-oper $(\cF_G,A,\cF_{B_+})$
  and any choice of $B_-$-reduction $\cF_{B_-}$ of $\cF_G$ preserved
  by the oper $q$-connection $A$, the resulting Miura $(G,q)$-oper is
  $Z'$-twisted for a particular $Z' \in W_G \cdot Z$.\\
2) If $Z$ is a regular element of $G$, then for each $Z' \in W_G \cdot
Z$, there is a unique $B_+$-reduction on the $(G,q)$-oper
$(\cF_G,A,\cF_{B_+})$ making it into a $Z'$-twisted Miura
$(G,q)$-oper.
\end{Prop}

From now on we will focus on the case when $Z$ is regular (see \cite{Frenkel:2020} for the statement of the Proposition \ref{Z prime} in the general situation). 

Consider a Miura $(G,q)$-oper with regular singularities determined by
polynomials $\{ \Lambda_i(z) \}_{i=1,\ldots,r}$. Thus, the underlying $(G,q)$-connection can be written in
the form \eqref{form of A}. We can define the associate $(H,q)$ connection which has the form:
\begin{equation}    \label{AH}
A^H(z)=\prod_ig_i(z)^{-\check{\alpha}_i}.
\end{equation}
We call $A^H(z)$ the \emph{associated Cartan $q$-connection} of the
Miura $q$-oper $A(z)$. 
If our Miura $q$-oper is $Z$-twisted, then we also have $A(z)=v(qz)Z v(z)^{-1}$, where
$v(z)\in B_-(z)$.  Since $v(z)$ can be written as
\begin{equation}    \label{vz}
v(z)=
\prod_i y_i(z)^{-\check{\alpha}_i} n(z), \qquad n(z)\in N_-(z), \quad
y_i(z) \in \C(z)^\times,
\end{equation}
Thus 
\begin{equation}    \label{giyi}
g_i(z)=\zeta_i\frac{y_i(qz)}{y_i(z)}\,.
\end{equation}

\subsection{ Miura-Pl\"ucker $(G,q)$-opers}\label{Sec:MiuraPlucker}
In this section, we will relax $Z$-twisted condition on the Miura $(G,q)$-opers.
  
Consider $\omega_i$, the $i$th fundamental weight of $G$.
Let $V^-_i$ be the irreducible representation of $G$ with lowest weight
$-\omega_i$ with respect to $B_-$. It comes equipped with a line $L^-_i
\subset V^-_i$ of lowest weight vectors stable under the action of
$B_+$. Let
$\nu^-_{\omega_i}$ be a generator of the line $L^-_i \subset V_i^-$. It is a
vector of weight $-\omega_i$ with respect to our maximal torus $H
\subset B_-$. The subspace of $V^-_i$ of weight $-\omega_i+\alpha_i$
is one-dimensional and is spanned by $e_i \cdot
\nu^-_{\omega_i}$. Therefore, the two-dimensional subspace $W_i$ of
$V^-_i$ spanned by the weight vectors $\nu^-_{\omega_i}$ and $e_i \cdot
\nu^-_{\omega_i}$ is a $B_-$-invariant subspace of $V^-_i$.

Now, let $(\cF_G,A,\cF_{B_+},\cF_{B_-})$ be a Miura $(G,q)$-oper
with regular singularities determined by polynomials $\{ \Lambda_i(z)
\}_{i=1,\ldots,r}$ (see Definition \ref{MiuraRS}). Recall that
$\cF_{B_-}$ is a $B_-$-reduction of a $G$-bundle $\cF_G$ on $\P^1$
preserved by the $(G,q)$-connection $A$. Therefore for each
$i=1,\ldots,r$, the vector bundle
$$
\cV^-_i = \cF_{B_-} \underset{B_-}\times V^-_i = \cF_G \underset{G}\times
V^-_i
$$
associated to $V^-_i$ contains a rank-two
subbundle
$$
\cW_i = \cF_{B_-} \underset{B_-}\times W_i
$$
associated to $W_i \subset V^-_i$, and $\cW_i$ in
turn contains a line subbundle
$$
\cL^-_i = \cF_{B_-} \underset{B_-}\times L^-_i
$$
associated to $L^-_i \subset W_i$.

Denote by $\phi_i(A)$ the $q$-connection on the vector bundle $\cV^-_i$
(or equivalently, a $(\GL(V_i),q)$-connection) corresponding to the
above Miura $q$-oper connection $A$. Since $A$ preserves $\cF_{B_-}$
(see Definition \ref{Miura}), we see that $\phi_i(A)$ preserves the
subbundles $\cL^-_i$ and $\cW_i$ of $\cV^-_i$. Denote by $A_i$ the
corresponding $q$-connection on the rank 2 bundle $\cW_i$.

Let us trivialize $\cF_{B_-}$ on a Zariski open subset of $\P^1$ so
that $A(z)$ has the form \eqref{form of A} with respect to this
trivialization. This trivializes the
bundles $\cV^-_i$, $\cW_i$, and $\cL^-_i$, so that the $q$-connection
$A_i(z)$ becomes a $2 \times 2$ matrix whose entries are in $\C(z)$, which can be thought of as Miura $(GL(2),q)$-opers.

Using this collection $\{A_i\}_{i=1, \dots, r}$ we introduce the notion of 
$Z$-{\em twisted Miura-Pl\"ucker $(G,q)$-oper}.

\begin{Def}  \cite{Frenkel:2020}   \label{ZtwMP}
  A $Z$-{\em twisted Miura-Pl\"ucker $(G,q)$-oper} is a meromorphic
  Miura $(G,q)$-oper on $\P^1$ with underlying $q$-connection $A(z)$
  satisfying the following condition: there exists $v(z) \in B_-(z)$
  such that for all $i=1,\ldots,r$, the Miura $(\GL(2),q)$-opers
  $A_i(z)$ can be
  written in the form
\begin{equation}    \label{gaugeA3}
A_i(z) = v(zq) Z v(z)^{-1}|_{W_i} = v_i(zq)Z_iv_i(z)^{-1},
\end{equation}
where $v_i(z) = v(z)|_{W_i}$ and $Z_i = Z|_{W_i}$.
\end{Def}

Thus the resulting Cartan connection (\ref{AH}) of $Z$-{\em twisted Miura-Pl\"ucker $(G,q)$-oper} as in the $Z$-twisted case:   
\begin{equation}
\label{AH1}
A^H(z)=\prod_i
\Bigg[\zeta_i\frac{y_i(qz)}{y_i(z)}\Bigg]^{-\check{\alpha}_i} 
\end{equation}

Next we will formulate the nondegeneracy condition for $Z$-{\em twisted Miura-Pl\"ucker $(G,q)$-opers} (see \cite{Frenkel:2020} for the equivalent definitions).

\begin{Def}\cite{Frenkel:2020} 
We call $Z$-{\em twisted Miura-Pl\"ucker $(G,q)$-oper} with regular singularities {\it nondegenerate} 
if each $y_i(z)$ from formula \eqref{AH1} is a polynomial, and for all $i,j,k$ with $i\ne j$ and $a_{ik} \neq 0, a_{jk} \neq 0$, the zeros of $y_i(z)$ and $y_j(z)$ are
  $q$-distinct from each other and from the zeros of
  $\Lambda_k(z)$.
\end{Def}

We notice that $\{y_i(z)\}_{i=1,\dots, r}$ can be chosen to be monic.

\subsection{$QQ$-systems and Miura-Pl\"ucker $(G,q)$-opers}
From now on, we will assume that our element $Z = \prod_i
\zeta_i^{\check\alpha_i} \in H$ satisfies the following property:
\begin{equation}    \label{assume}
\prod_{i=1}^r \zeta_i^{a_{ij}} \notin q^\Z, \qquad
\forall j=1,\ldots,r\,.  
\end{equation}
Since $\prod_{i=1}^r \zeta_i^{a_{ij}}\ne 1$ is a special case of
\eqref{assume}, this implies that $Z$ is {\em regular semisimple}.

We introduce the following system of equations \cite{Frenkel:2020}:
\begin{multline}\label{qq}
\wt{\xi}_iQ^i_{-}(z)Q^i_{+}(qz)-\xi_iQ^i_{-}(qz)Q^i_{+}(z) = \\
\Lambda_i(z)\prod_{j> i}\Big[Q^j_{+}(qz)\Big]^{-a_{ji}}
\prod_{j< i}\Big[Q^j_{+}(z)\Big]^{-a_{ji}}, \qquad
i=1,\ldots,r,
\end{multline}
where
\begin{equation}    \label{xi}
\wt{\xi}_i=\zeta_i \prod_{j>i} \zeta_j^{a_{ji}}, \qquad
{\xi}_i=\zeta^{-1}_i\prod_{j< i} \zeta_j^{-a_{ji}}
\end{equation}
and we use the ordering of simple roots from the definition of
$(G,q)$-opers.

We call this the $QQ$-{\em system} associated to $G$ and a collection
of polynomials $\Lambda_i(z)$, $i=1,\ldots,r$.

A polynomial solution $\{ Q^i_+(z),Q^i_-(z) \}_{i=1,\ldots,r}$ of
\eqref{qq} is called {\em nondegenerate} if it has the following
properties: condition \eqref{assume} holds for the $\zeta_i$'s; for
all $i,j,k$ with $i \neq j$ and $a_{ik}, a_{jk} \neq 0$, the
zeros of $Q^j_+(z)$ and $Q^j_-(z)$ are $q$-distinct from each other
and from the zeros of $\Lambda_k(z)$; and the polynomials $Q^i_+(z)$
are monic.

We have the following Theorem, which relates solutions of the $QQ$-system to 
$Z$-twisted Miura-Pl\"ucker $(G,q)$-opers.

\begin{Thm}   \cite{Frenkel:2020}  \label{inj}
  There is a one-to-one correspondence between the set of
  nondegenerate $Z$-twisted Miura-Pl\"ucker $(G,q)$-opers and the set
  of nondegenerate polynomial solutions of the $QQ$-system \eqref{qq}.
\end{Thm}

The proof of this theorem relies on solving explicitly the conditions in (\ref{gaugeA3}), 
using the reduction to the $B_+$-invariant two-dimensional subspaces $W_i$, which makes the element $v(z)$ diagonalizing $A(z)$ look as follows:

\begin{equation}    \label{vdots}
v(z) = \prod_{i=1}^r \Big[Q^i_{+}(z)\Big]^{-\check\alpha_i} \prod_{i=1}^r
e^{-\frac{Q^i_{-}(z)}{Q^i_{+}(z)} f_i} \dots ,
\end{equation}
where $Q_i^{\pm}(z)$ are the solutions of the $QQ$-system and 
the dots stand for the exponentials of higher commutator terms 
in ${\mathfrak n}_-=\Lie N_-$. 
As a result, the q-connection $A(z)$ for the Miura-Pl\"ucker $(G,q)$-oper can be expressed as follows:
\begin{align}    \label{key}
A(z) &=\prod_j\left[ \zeta_j\frac{Q^j_+(qz)}{Q^j_+(z)}
\right]^{-\check{\alpha}_j} e^{\frac{\Lambda_j(z) Q^j_+(z)}{\zeta_j
    Q^j_+(qz)}f_i} \\ &= \prod_j
\Big[\zeta_jQ_{+}^j(qz)\Big]^{-\check{\alpha}_j}
e^{\frac{\Lambda_j(z)}{\zeta_j Q_+^j(qz)Q_+^j(z)}e_j}
\Big[{Q_{+}^j(z)}\Big]^{\check{\alpha}_j}\,.
    \label{key1}
\end{align}

\subsection{Bethe ansatz equations}

The $QQ$-system \eqref{qq} gives rise to a system of
equations only involving the $Q_+^i(z)$'s. Let $\{ w^k_i
\}_{k=1,\ldots,m_i}$ be the set of roots of the polynomial
$Q^i_+(w)$.  We call the system of equations
\begin{equation}    \label{bethe}
\frac{Q_+^{i}(qw^k_i)}{Q_+^{i}(q^{-1}w^k_i)} \prod_j\zeta_j^{a_{ji}} =
- \; \; \frac{\Lambda_i(w_k^i)\prod_{j>
  i}\Big[Q^j_{+}(qw_k^i)\Big]^{-a_{ji}}\prod_{j<
  i}\Big[Q^j_{+}(w_k^i)\Big]^{-a_{ji}}}{\Lambda_i(q^{-1}w_k^i)\prod_{j>
  i}\Big[Q^j_{+}(w_k^i)\Big]^{-a_{ji}}\prod_{j<
  i}\Big[Q^j_{+}(q^{-1}w_k^i)\Big]^{-a_{ji}}}
\end{equation}
for $i=1,\ldots,r$, $k=1,\ldots,m_i$ the {\em Bethe Ansatz equations}
for the group $G$ and the set $\{ \Lambda_i(z) \}_{i=1,\ldots,r}$.

For simply laced $G$, this system is equivalent to the system of Bethe
Ansatz equations that appear in the $U_q \ghat$ XXZ-type
model~\cite{OGIEVETSKY1986360,RW,Reshetikhin:1987}. In case of 
non-simply laced $G$, we obtain a different system of Bethe Ansatz
equations, which, as far as we know, has not yet been studied in the
literature on quantum
integrable systems. An additive version of this system appeared earlier in
\cite{Mukhin_2005}. As will be explained in \cite{FHRnew},
these Bethe Ansatz equations correspond to a novel quantum integrable
model in which the spaces of states are representations of the twisted
quantum affine Kac-Moody algebra $U_q {}^L\ghat$, where $^L\ghat$ is
the Langlands dual Lie algebra of $\ghat$.

The following Theorem is true (see \cite{Frenkel:2020} for details).

\begin{Thm}    \label{BAE}
  There is a bijection between two sets: the nondegenerate
  polynomial solutions of the $QQ$-system \eqref{qq} and the nondegenerate solutions of Bethe
  Ansatz equations \eqref{bethe}.
\end{Thm}

\subsection{B\"acklund transformations, $Z$-twisted condition and the full $QQ$-system} 

Now we will relate Miura-Pl\"ucker to $Z$-twisted $(G,q)$-opers provided some conditions on the $QQ$-system are satisfied. 

First, we introduce a set of certain transformations on Miura-Pl\"ucker $(G,q)$-opers, which we refer to as {\it B\"acklund transformations}.

\begin{Prop}    \label{fiter}
  Consider the $q$-gauge transformation of the $q$-connection $A$
  given by formula \eqref{key}:
\begin{eqnarray}
A \mapsto A^{(i)}=e^{\mu_i(qz)e_i}A(z)e^{-\mu_i(z)e_i},
\quad \operatorname{where} \quad \mu_i(z)=\frac{\prod\limits_{j\neq
    i}\Big[Q_+^j(z)\Big]^{-a_{ji}}}{Q^i_{+}(z)Q^i_{-}(z)}\,.
\label{eq:PropDef}
\end{eqnarray}
Then $A^{(i)}(z)$ can be obtained from $A(z)$ by
substituting in formula \eqref{key} (or \eqref{key1})
\begin{align}
Q^j_+(z) &\mapsto Q^j_+(z), \qquad j \neq i, \\
Q^i_+(z) &\mapsto Q^i_-(z), \qquad Z\mapsto s_i(Z)\,.
\label{eq:Aconnswapped}
\end{align}
\end{Prop}

By construction, $A^{(i)}(z)$ is an $s_i(Z)$-twisted Miura-Pl\"ucker
$(G,q)$-oper, which corresponds to the polynomials $\{ \wt{Q}^j_+(z)
\}_{j=1,\ldots,r}$, where $\wt{Q}^+_j(z)=Q^+_j(z)$ for $j \neq i$ and
$\wt{Q}^+_i(z) = Q^-_i(z)$. The conditions for it to be nondegenerate
are spelled out in the following lemma (see \cite{Frenkel:2020}).

\begin{Lem} \label{nondegcond} Suppose that the roots of the
  polynomial $Q^i_-(z)$ constructed in the proof of Theorem \ref{BAE}
  are $q$-distinct from the roots of $\Lambda_k(z)$ for $a_{ik} \neq
  0$ and from the roots of $Q^j_+(z)$ for $j\neq i$ and $a_{jk} \neq
  0$.  Then, the data
\begin{align} \label{qqm}
  \{ \wt{Q}^j_+ \}_{j=1,\ldots,r} &= \{ Q^1_{+}, \dots,
  Q^{i-1}_+,Q^i_-,Q^{i+1}_+ \dots , Q^r_{+} \}; \\ \notag
  \{ \wt{\zeta}_j \}_{j=1,\ldots,r} &= \{
  \zeta_1,\dots,\zeta_{i-1},\zeta_i^{-1}\prod\limits_{j\neq
  i}\zeta_j^{-a_{ji}},\dots,\zeta_r\}
\end{align}
give rise to a nondegenerate solution of the Bethe Ansatz equations
\eqref{bethe} corresponding to $s_i(Z) \in H$.  Furthermore, there
exist polynomials $\{ \wt{Q}^j_- \}_{j=1,\ldots,r}$ that together with
$\{ \wt{Q}^j_+ \}_{j=1,\ldots,r}$ give rise to a nondegenerate
solution of the $QQ$-system \eqref{qq} corresponding to $s_i(Z)$.
\end{Lem}

Thus, if the conditions of Lemma \ref{nondegcond} are
satisfied, we can associate to every nondegenerate $Z$-twisted
Miura-Pl\"ucker $(G,q)$-oper a nondegenerate $s_i(Z)$-twisted
Miura-Pl\"ucker oper via $A(z) \mapsto A^{(i)}(z)$. We call this
procedure a {\em B\"acklund-type transformation} associated to the
$i$th simple reflection of the Weyl group $W_G$. We now
generalize this transformation to other Weyl group elements.
We denote the elements of the $QQ$-system corresponding to the reflection $w$ as
\begin{eqnarray}
\{Q^{i,w}_+(z)\}_{i=1,\dots, r}
\end{eqnarray}

\begin{Def}
We call the system of equations generated by $\{Q^{i,w}_+(z)\}_{i=1,\dots, r}$ for all $w\in W_G$ the {\it full $QQ$-system}.  A solution of the $QQ$-system
  \eqref{qq} is called {\it $W_G$-generic} if 
by consecutively applying the procedure described in Lemma \ref{nondegcond} for $w=  s_{i_1} \dots s_{i_k}$, we obtain a system of nondegenerate solutions of the $QQ$-systems corresponding to the  elements $w_j(Z) \in H$, where $w_k=s_{i_{k-j+1}} \ldots s_{i_k}$
  with $j=1,\ldots,k$. A $Z$-twisted Miura-Pl\"ucker $(G,q)$-oper is called $W_G$-generic if it corresponds to a $W_G$-generic solution of the $QQ$-system via the bijection in Theorem \ref{inj}.
\end{Def}

One can show that for $W_G$-{\em generic} Miura-Pl\"ucker $(G,q)$-oper there exists $b_-(z)\in B_-(z)$  such that:
\begin{equation}    \label{Abplus}
A(z) = b_-(qz) Z b_-(z)^{-1}.
\end{equation}
Then the following Theorem holds (see \cite{Frenkel:2020}).
\begin{Thm}    \label{w0}
Every $W_G$-generic $Z$-twisted Miura-Pl\"ucker $(G,q)$-oper is a
nondegenerate $Z$-twisted Miura $(G,q)$-oper.
\end{Thm}

\begin{Rem}
One should note that for a given $Z$-twisted $(G,q)$-oper, the corresponding $\{Q^{i,w}_+(z)\}_{i=1,\dots, r} $-systems are in 1-to-1 correspondence with the corresponding $Z$-twisted $Z$-twisted Miura $(G,q)$-opers for regular semisimple $G$ (see Proposition \ref{Z prime}). That immediately leads to the fact that the transformations (\ref{qqm}) preserve Weyl group relations.
\end{Rem}

In the case of $G=SL(r+1)$, there is an alternative construction of $Z$-twisted Miura opers using the determinant formulae: see e.g. \cite{KSZ} and more complete description in \cite{Koroteev:2020tv}. Using this approach one can show that nondegenerate $Z$-twisted Miura-Pl\"ucker $(SL(r+1),q)$-oper is $Z$-twisted Miura $(SL(r+1),q)$-oper.

\vskip.1in

In \cite{Brinson:2021ww}, by following an analogous prescription, the differential version of $G$-opers was considered. It was shown that non-degenerate $Z$-twisted Miura-Pl\"ucker $G$-opers are in fact $Z$-twisted Miura $G$-opers as long as the degrees of polynomials $\{\Lambda_{i}(z), Q^{i,w}_+(z)\}^r_{i=1, w\in W}$ the full $QQ$-system allow for polynomial solutions of the $QQ$-system. 
The argument follows the expansion of solutions of the $QQ$-system around $\widetilde\xi_i= 0$ in \eqref{xi}. In this regime, the $QQ$-system reduces to a system of elementary polynomial relations, which is non-degenerate for distinct roots of $\Lambda_i$. Then one analytically continues with respect to $\widetilde\xi_i$. The argument suits both the differential and $q$-difference cases equally well.

\vskip.1in

Therefore, we formulate the following statement.
\begin{Thm}
Any nondegenerate $Z$-twisted Miura-Pl\"ucker $(G,q)$-oper is $Z$-twisted Miura $(G,q)$-oper if the degrees of the full $QQ$-system $\{\Lambda_{i}(z), Q^{i,w}_+(z)\}^r_{i=1, w\in W}$ allow the polynomial solution.  
\end{Thm}

\section{$Z$-twisted $q$-opers and $(G,q)$-Wronskians}\label{Sec:GeneralizedWronskians}

\subsection{Generalized Minors and Pl\"ucker coordinates}
In the next section, we will discuss another approach to Miura $(G,q)$-opers. This approach is based on the datum of the corresponding connection in the set of fundamental representations. There is a way to encode this datum in terms of certain explicit ``coordinates" one can associate to a group element. 
These coordinates are the generalizations of minors for $SL(N)$. They were used by Berenstein, Fomin and Zelevinsky in the study of Schubert cells and double Bruhat cells in the combinatorial context of cluster algebras. In the seminal paper, \cite{FZ} the generalized minors appeared as a set of parameters the sign of which determines the {\it total positivity} of elements from double Bruhat cells.

Let us define what generalized minors are. Consider the big cell in Bruhat decomposition: $G_0=N_-HN_+$, where we remind that $G$ is a simple simply-connected Lie group.
For a given element $g\in G_0$ we can write it as 
\begin{eqnarray}
g=n_-~h~n_+.
\end{eqnarray}
Let $V^+_i$ be the irreducible representation of $G$ with highest weight $\omega_i$ and highest weight vector $\nu^+_{\omega_i}$ which is the eigenvector for any $h\in H$, i.e. $h\nu^+_{\omega_i}=[h]^{\omega_i}\nu^+_{\omega_i}$, $[h]^{\omega_i}\in \mathbb{C}^{\times}$.
Note that $V^+_i$ is isomorphic to one space from the family $\{V^-_i\}^r_{i=1}$. 
Let us introduce the following definition:
\begin{Def}\cite{FZ}
The following regular functions $\{\Delta^{\omega_i}\}_{i=1, \dots, r}$ on $G$, whose values on a dense set $G_0$ are given  
\begin{eqnarray}
\Delta^{\omega_i}(g)=[h]^{\omega_i}, \quad i=1, \dots, r
\end{eqnarray}
will be referred to as  {\it principal minors} of a group element $g$. 
\end{Def}

In the case of $G=SL(N)$ these functions stand for principal minors of the standard matrix realization of $SL(N)$.

Other generalized minors are obtained by the action of the Weyl group elements on the left and the right of $g$ and then applying the appropriate lifts of Weyl group elements $u,v$ on the right and the left and then applying principal minors to the result.

Namely, we have the following
\begin{Def}
For $u, v\in W_G$, we define a regular function $\Delta_{u\omega_i, v\omega_i}$ on $G$ by setting
\begin{equation}
\Delta_{u\omega_i, v\omega_i}(g)=\Delta^{\omega_i}(\tilde{u}^{-1}g\tilde{v}).
\end{equation}
\end{Def}
 Notice that in this notation $\Delta_{\omega_i, \omega_i}(g)=\Delta^{\omega_i}(g)$. Consider the orbit $\mathcal{O}_{W_G}=W_G\cdot \mathbb{C}\nu^+_{\omega_i}$, This way we have the following Proposition.\\

\begin{Prop}
The action of the group element $g$ on the highest weight vector $\nu^+_{\omega_i}\in V^+_i$ is given by: 
\begin{eqnarray}
g\cdot \nu^+_{\omega_i}=\sum_{w\in W}\Delta_{w\cdot \omega_i, \omega_i}(g)\tilde{w}\cdot\nu^+_{\omega_i}+\dots,
\end{eqnarray}
where dots stand for the vectors, which do not belong to the orbit $\mathcal{O}_W$.
\end{Prop}

The set of generalized minors $\{\Delta_{w\cdot \omega_i, \omega_i}\}_{w\in W; i=1, \dots, r}$ creates a set of coordinates on $G/B^+$, known as {\it generalized Pl\"ucker coordinates}. In particular, the set of zeroes of each of $\Delta_{w\cdot \omega_i, \omega_i}$ is a uniquely and unambiguously defined hypersurface in $G/B$. This feature is important for characterizing Schubert cells as quasi-projective subvarieties of a generalized flag variety, see \cite{Fomin:wy} for details. We will need the following Corollary.
\begin{Cor}
If the collection $\{\Delta_{w\cdot \omega_i, \omega_i}(g)\}_{w\in W; i=1, \dots, r}$ does not have vanishing elements, then $g\in B_+w_0B_+$. 
\end{Cor}

One of the first consequences of the formalism of generalized minors is the following Theorem.
\begin{Thm}
For a nondegenerate $Z$-twisted Miura-Pl\"ucker $(G,q)$-oper with $q$-connection $A(z)=v(qz)Zv(z)^{-1}$,  where $v(z)\in B_-(z)$ we have the following relation:
\begin{equation}
\Delta_{w\cdot \omega_i, \omega_i}(v^{-1}(z))=Q_+^{w,i}(z)
\end{equation}
for any $w\in W$.
\end{Thm}
\begin{proof}
Notice that 
$\Delta^{\omega_i}(v^{-1}(z))=Q_+^{i}(z)$. 
Indeed, following (\ref{vdots}), we have: 
$$v^{-1}(z) = \prod_{i=1}^r
e^{\frac{Q^i_{-}(z)}{Q^i_{+}(z)} f_i} \prod_{i=1}^r \Big[Q^i_{+}(z)\Big]^{\check\alpha_i} 
 \dots\,,$$
where dots stand for exponentials of higher commutators of $\{f_i\}$, we obtain that 
\begin{eqnarray}\label{odeim}
v^{-1}(z)\nu^+_{\omega_i}=Q_+^i(z)\nu^+_{\omega_i} +Q_-^i(z)f_i\nu^+_{\omega_i}+\dots\,, 
\end{eqnarray}
where dots stand for the vectors of lower weights.

Now take into account that $v(z)\tilde{w}^{-1}=u_+(z)v_w(z)$, where 
$u_+(z)\in N_+(z), ~v_{w}(z)\in B_-(z)$. 
Here $v_w(z)$ is the trivializing element for 
$A^w(z)=v_w(z){w}(Z)v_{w}^{-1}(z)$. 
This means that $\Delta^{\omega_i}(v_w^{-1}(z))=Q_+^{w,i}(z)$, which is obtained by B\"acklund transformations. 
Therefore, generalized minors satisfy the relation  
$\Delta_{w\cdot \omega_i, \omega_i}(v^{-1}(z))=Q_+^{w,i}(z)$.
\end{proof}

Following Theorem 1.12 of \cite{FZ} we obtain the following Corollary:

\begin{Cor}
The minors $\Delta_{w\cdot \omega_i, \omega_i}$ uniquely determine the element $v^{-1}(z)$.
\end{Cor}

We started this section with the explicit definition of the principal minors by means of Gaussian decomposition. The following proposition (see Corollary 2.5 in \cite{FZ}) provides a necessary and sufficient condition of its existence for a given group element.

\begin{Prop}
An element $g\in G$ admits the Gaussian decomposition if and only if $\Delta^{\omega_i}(g)\neq 0$ for any $i=1, \dots, r$.
\end{Prop}

Finally, we end this section with the fundamental relation (\cite{FZ}, Theorem 1.17) between generalized minors, which we will relate to the $QQ$-systems. 

\begin{Prop}
Let, $u,v\in W$, such that for 
$i\in \{1, \dots, r\}$,  $\ell(uw_i)=\ell(u)+1$,  $\ell(vw_i)=\ell(v)+1$. Then 
\begin{equation}\label{eq:minorsgen}
\Delta_{u\cdot\omega_i, v\cdot\omega_i}\Delta_{uw_i\cdot \omega_i, vw_i\cdot\omega_i}-
\Delta_{uw_i\cdot\omega_i, v\cdot\omega_i}\Delta_{u\cdot\omega_i, vw_i\cdot\omega_i}=\prod_{j\neq i}\Delta_{u\cdot\omega_j, v\cdot\omega_j}^{-a_{ji}}, 
\end{equation}
\end{Prop}

\subsection{$(G,q)$-Wronskians and Generalized Minors}
First, we introduce a notion of $(G,q)$-Wronskian which, as we will see later, under certain nondegeneracy conditions, is  
equivalent to the definition of $Z$-twisted Miura $(G,q)$-oper.

Let $V^+_i$ be the irreducible representation of $G$ with highest weight
$\omega_i$ with respect to $B_+$. It comes equipped with a line $L^+_i
\subset V^+_i$ of the highest weight vectors stable under the action of
$B_+$. Let 
$\nu^+_{\omega_i}$ be a generator of the line $L^+_i \subset V_i^+$. It is a
vector of weight $\omega_i$ with respect to our maximal torus $H
\subset B_-$. The subspace $L^{+}_{c,i}$ of $V_i$ of weight $c^{-1}\cdot\omega_i$
is one-dimensional and is spanned by ${s}^{-1}
\nu^+_{\omega_i}$.

Suppose we have a principal $G$-bundle $\cF_G $ and its $B_+$-reduction $\cF_{B_+} $ and thus an $H$-reduction  $\cF_{H} $ as well.
Therefore for each
$i=1,\ldots,r$, the vector bundle
$$
\cV^+_i = \cF_{B_+} \underset{B_+}\times V^+_i = \cF_G \underset{G}\times
V^+_i
$$
associated to $V^+_i$ contains an $H$-line subbundles
$$
\cL^+_i = \cF_{H} \underset{H}\times L^+_i, \quad \cL^{+}_{c,i} = \cF_{H} \underset{H}\times L^{+}_{c,i}
$$
associated to $L^+_i, L^{+}_{c,i} \subset V^+_i$.

Consider a meromorphic section $\mathscr{G}$ of $\mathcal{F}_G$. It is a section of $\mathcal{F}_G$ on $U$, a Zariski dense set of $\mathbb{P}^1$. Given the fact that one can always choose $U$, so that restriction of $\mathcal{F}_G$ to $U$ is a trivial $G$-bundle, one can express this section as an element $\mathscr{G}(z)\in G(z)$.

\begin{Def}    \label{genqwr}
The {\it $(G,q)$-Wronskian} on $\mathbb{P}^1$ is the quadruple $(\cF_G,\cF_{B_+}, \mathscr{G}, Z)$, where $\mathscr{G}$ is a meromorphic section of a principle bundle $\cF_G$, $\cF_{B_+}$ is a reduction of $\cF_G$ to $B_+$, $Z\in H=B_+/[B_+,B_+]$, satisfying the following condition.
There exist a Zariski open dense subset $U\subset \mathbb{P}^1$ 
together with the trivialization $\imath_{B_+}$ of $\cF_{B_+}$, so that 
for certain $\{v_i^+, v_{c,i}^+\}_{i=1, \dots, r}$ which are the sections of line bundles  $\{\mathcal{L}^+_i, \mathcal{L}^{+}_{c,i}\}_{i=1, \dots, r}$ on 
$U \cap M_q^{-1}(U)$ we have $\mathscr{G}$ as an element of $G(z)$ satisfy the following condition:
\begin{eqnarray}
\mathscr{G}^q \cdot v_i^{+}=Z\cdot \mathscr{G}\cdot v_{c,i}^+, 
\end{eqnarray}
where the superscript $q$ stands for the pull-back of the corresponding section with respect to the map $M_q$.
\end{Def}

Effectively, the definition implies that  
there exists a Zariski open dense subset $U \subset \P^1$ together with a trivialization $\imath_{B_+}$ of $\mathcal{F}_{B_+}$ such that the restriction of $\mathscr{G}$ to $U \cap M_q^{-1}(U)$ written as an element of $G(z)$ satisfies the following conditions
\begin{eqnarray}\label{eq:qwr}
Z^{-1}\mathscr{G}(qz)~ \nu^+_{\omega_i}=\mathscr{G}(z)\cdot s_{\phi}(z)^{-1}\cdot\nu^+_{\omega_i},
\end{eqnarray}
where $s_\phi(z)=\prod_i\phi_i^{-\check{\alpha}_i}s_i$ is a lift of the Coxeter element $c\in W$ to $G(z)$, which is fixed for all $i\in \{1, \dots, r\}$. 

It is clear that the structure of the $(G,q)$-Wronskian depends on the generalized minors of $\mathscr{G}(z)$ through the action of $\mathscr{G}(z)$ on $\nu^+_i$ and the choice of the lift $s_{\phi}(z)$, which through the coefficients $\{\phi_i(z)\}_{i=1,\dots, r}$ depends on the choice of the sections $\{v_i^+, v_{c,i}^+\}_{i=1, \dots, r}$.

The following two definitions clarify the type of objects will restrict the type of $(G,q)$-Wronskians we will study in this paper.

\begin{Def}
$(G,q)$-Wronskian has {\it regular singularities} if 
\begin{eqnarray}\label{eq:Invordering}
s_{\Lambda}(z)^{-1}=\prod^{\rm inv}_{i}s_i\Lambda_i^{\check{\alpha}_i},
\end{eqnarray}
where $\{\Lambda_i\}_{i=1,\dots, r}$ are polynomials, and the superscript ``inv" stands for the inverse order to the ordering in the Coxeter element $c$.
\end{Def}

Let us give a simple example of $(G,q)$-Wronskian in the lowest rank case and give the relation to the $QQ$-system.\\

\noindent {\bf Example.}
Let $G=SL(2)$ then \eqref{eq:qwr} reads 
\begin{equation}\label{eq:Wdiffeqsl2}
\mathscr{G}(q z)\nu^+_{\omega}=Z \mathscr{G}(z)s^{-1}(z)\nu^+_{\omega}\,.
\end{equation}
In this case 
\begin{equation}
s^{-1}(z)=\tilde s^{-1}\Lambda(z)^{\check{\alpha}}= \begin{pmatrix} 0&\Lambda(z)^{-1} \\ \Lambda(z)&0\end{pmatrix}\,,\qquad \nu_\omega^+=\begin{pmatrix} 1\\0\end{pmatrix}\,,\qquad Z= \begin{pmatrix} \zeta&0\\0&\zeta^{-1}\end{pmatrix}\,.
\end{equation}
One can immediately see that 
\begin{equation}
\mathscr{W}(z)=\begin{pmatrix} \Delta_{\omega_i, \omega_i}(\mathscr{G}(z)) & \Delta_{\omega_i, s^{-1}\omega_i}(\mathscr{G}(z))\\ \Delta_{s\omega_i,  \omega_i}(\mathscr{G}(z))& \Delta_{s\omega_i, s^{-1}\omega_i}(\mathscr{G}(z))\end{pmatrix}
\end{equation}
satisfies \eqref{eq:Wdiffeqsl2} provided that the following relations take place:
\begin{equation}
\Delta_{\omega_i, \omega_i}(\mathscr{G}(qz))=\zeta \Lambda(z)\Delta_{s\omega_i, \omega_i}(\mathscr{G}(z))\,,\qquad \Delta_{\omega_i,  s^{-1}\omega_i}(\mathscr{G}(qz)) = \zeta^{-1}\Lambda(z)\Delta_{s\omega_i,  s^{-1}\omega_i}(\mathscr{G}(z))\,.
\end{equation}
Using the identification $\Delta_{\omega_i, \omega_i}(\mathscr{G}(z))\equiv Q^+(z)$, $\Delta_{s\omega_i,  \omega_i}(\mathscr{G}(z))\equiv Q^-(z)$, we obtain
\begin{equation}
\mathscr{G}(z)=\begin{pmatrix} Q^+(z)& \zeta^{-1} \Lambda(z)^{-1} Q^+(qz)\\ Q^-(z) & \zeta \Lambda(z)^{-1} Q^-(qz)\end{pmatrix}\,,
\end{equation}
where we put $Q_+(z)=Q^\omega(z)$ and $Q_-(z)=Q^{s\omega}(z)$ according to the notations from our previous papers. Notice that the condition that $\mathscr{G}(z)\in SL(2)$, i.e.  det$\mathscr{G}(z)=1$ leads to the $QQ$-equation:
$$
\zeta Q^+(z) Q^-(qz)-\zeta^{-1} Q^+(qz) Q^-(z) = \Lambda(z).
$$

This example justifies the name $(G,q)$-Wronskian as a generalization of the q-Wronskian matrix.
However, we will see that in this form it is not uniquely defined for higher rank. Also, to get in touch with nondegenerate $QQ$-systems, we have to put relevant nondegeneracy conditions on $(G,q)$-Wronskians.

\begin{Def}
We say that  $(G,q)$-Wronskian with regular singularities is {\it nondegenerate} if  $\Delta_{w\cdot \omega_i,\omega_i}(\mathscr{G}(z))$ are nonzero polynomials for all $w\in W$ and $i=1,\dots, r$.  
For all $i,j,k$ with $i \neq j$ and $a_{ik}, a_{jk} \neq 0$, the
zeros of $\Delta_{\omega_i, \omega_i}$ and $\Delta_{s_i\omega_i, \omega_i}$ are $q$-distinct from each other, and also zeroes of $\Delta_{w\cdot \omega_i, \omega_i}$ are $q$-distinct from the zeros of $\{\Lambda_k(z)\}_{k=1, \dots, r}$ for all $i$.
\end{Def}

These definitions lead to the following Corollary.

\begin{Cor}
The nondegeneracy condition of $(G,q)$-Wronskian with regular singularities implies:
\begin{enumerate}
\item $\mathscr{G}(z)$ admits Gaussian decomposition: $\mathscr{G}(z)\in N_-(z)H(z)N_+(z)$, 
\item $\mathscr{G}(z)$ belongs to the largest Bruhat cell: $\mathscr{G}(z)\in B_+(z)w_0B_+(z)$. 
\end{enumerate}
\end{Cor}

\begin{proof}
Condition (2) implies first of all that $\Delta_{\omega_i,\omega_i}(\mathscr{G}(z))\neq 0$. That implies Gaussian decomposition according to Corollary 2.5 of \cite{FZ}. The second property follows from Proposition 3.3 of \cite{Fomin:wy}.
\end{proof}

An important property is a non-uniqueness of the $(G,q)$-Wronskian as defined by the generalized minors.

\begin{Prop}\label{nonuniqdef}
Given a solution $\mathscr{G}(z)$ of the equation (\ref{eq:qwr}), $\mathscr{G}(z)n^+(z)$ is a solution of (\ref{eq:qwr}) if and only if 
\begin{equation}
s~n_+(z)~s^{-1}\in N_+(z).
\end{equation}
\end{Prop}

Later we will eliminate this ambiguity and add more constraints than (\ref{eq:qwr}), but first we 
investigate its lower triangular part and relate it to $QQ$-system and $q$-opers.

Let us list another important property of $(G,q)$-Wronskian:
\begin{Prop}\label{weylwr}
For any $w\in W$ and the $(G,q)$-Wronskian $(\cF_G,\cF_{B_+}, \mathscr{G},Z)$ with regular singularities, 
the element $\tilde{w}\cdot\mathscr{G}(z)$ stands for $(G,q)$-Wronskian $(\cF_G,\cF_{B_+}, \mathscr{G}, w(Z))$ with the same regular singularities.
\end{Prop}
\begin{proof}
The proof is obtained by the direct application of $\tilde{w}$ to $\mathscr{G}(z)$ in (\ref{eq:qwr}).
\end{proof}

\subsection{Extended $QQ$-system for $(G,q)$-Wronskian}
In this subsection, we find the relation between $(G,q)$-Wronskians and the $QQ$-systems via the fundamental relation (\ref{eq:minorsgen}) applied to $\mathscr{G}(z)$.

First, we formulate a Proposition, which allows to reformulate a specific subset in the family of relations (\ref{eq:minorsgen}).

\begin{Prop}
Minors $\Delta_{\omega_i, \omega_i}$, $\Delta_{w_i\omega_i, c^{-1}\cdot\omega_i}$,  $\Delta_{w_i\cdot\omega_i, \omega_i}$, $\Delta_{\omega_i, c^{-1}\cdot\omega_i}$ satisfy the following relation:
\begin{align}\label{eq:sminors}
\Delta_{\omega_i, \omega_i}\Delta_{w_i\cdot\omega_i, c^{-1}\cdot \omega_i}-&\Delta_{w_i\cdot\omega_i, \omega_i}\Delta_{\omega_i, c^{-1}\cdot\omega_i}\nonumber\\
&=\prod_{j<{i=i_{l}}}\Delta_{\omega_j, c^{-1}\cdot\omega_j}^{-a_{ji}}
\prod_{j>{i=i_{l}}}\Delta_{\omega_j, \omega_j}^{-a_{ji}}, \qquad i=1,\dots, r,
\end{align}
where the ordering is taken with respect to the decomposition of 
${c}^{-1}=w_{i_1}\dots, w_{i_l}, \dots, w_{i_r}$.
\end{Prop}
\begin{proof}
To prove that let us apply the relation (\ref{eq:minorsgen}) to the case when $u=1$, $v=w_{i_1}w_{i_2}\dots w_{i_{l-1}}$, so that $w_{i_{l}}=w_i$. Then $v\cdot \omega_j=c^{-1}\cdot \omega_j$ if $j\le {i_l}=i$ and $v\cdot \omega_j=\omega_j$ if $j>i_l=i$ and the statement of the Proposition follows immediately. 
\end{proof}

To apply this set of relations to $\mathscr{G}(z)$ and make full use of the difference equation it satisfies, we will need the following technical Lemma.

\begin{Lem}

Let $c^{-1}=\prod^{\rm inv}_i w_i=w_{i_1}w_{i_2}\dots w_{i_r}$ corresponds to the  lift of the inverse Coxeter element to $G$. Then $s_{\Lambda}^{-1}(z)=\prod_i^{{\rm inv}}\left(s_i \,\Lambda_i^{\check{\alpha}_i}\right)$ can be expressed as follows:
\begin{equation}
\label{eq:SinvLiftCalculation}
s_{\Lambda}^{-1}(z)= {s}^{-1} \prod_i \Lambda_i^{d_i}\,,
\end{equation}
where $d_i = \sum_{j=1}^i d_{ij} \check{\alpha}_j$ and 
\begin{align}
d_{ij} &= a_{i,i-j}-\sum_{l=1}^{i-j-1} a_{i,i-l}\cdot a_{i-l,i-j} + \sum_{l>m} a_{i,i-l}\cdot a_{i-l,i-l-m}\cdot a_{i-l-m,i-j}-\dots \cr
&+(-1)^r a_{i,i-1}\cdot a_{i-1,i-2}\cdots a_{i-j-1,i-j}\,,
\end{align}
for $j<i$ and $d_{ii}=1$.

\end{Lem}

For instance, for $SL(r+1)$ with a standard ordering along the Dynkin diagram we have $d_i = \sum_{j=1}^i \check{\alpha}_j$.

Now let us apply that to the group element $\mathscr{G}(z)$ and obtain the following Proposition.
\begin{Prop}\label{propfundrel}
Let $\mathscr{G}(z)$ be a non-degenerate $(G,q)$-Wronskian with regular singularities parametrized by the polynomials $\{\Lambda_i(z)\}_{i=1,\dots, r}$. Then we have:

1) The fundamental relation (\ref{eq:qwr}) for  $\mathscr{G}(z)$ is equivalent to the relation 
\begin{equation}\label{eq:DefLF1}
\Delta_{w\cdot\omega_i, c^{-1}\cdot\omega_i}(\mathscr{G}(z))=\Bigg[\prod_j\zeta^{\langle\check{\alpha}_j, w\cdot \omega_i\rangle}_{j}\Bigg]F_i(z)\Delta_{w\cdot \omega_i, \omega_i}(\mathscr{G}(qz))
\end{equation}
for any $w\in W$, where the proportionality coefficients $F_i(z)$ depend on $Z$ and the lift of the Coxeter element to $G(z)$ only:
\begin{equation}\label{eq:DefLF}
F_i(z)=L_i(z)^{-1}\,,\qquad  L_i(z)=\Lambda_i (z)^{d_{i,i}}\cdot \Lambda_{i-1}(z)^{d_{i-1,i}}\cdots \Lambda_1(z)^{d_{1,i}}\,,
\end{equation}

2) The relations between minors (\ref{eq:sminors}) can be written as follows in terms of 
$\{\Delta_{\omega_i, \omega_i}\}_{i=1,\dots, r}$,  $\{\Delta_{w_i\omega_i, \omega_i}\}_{i=1,\dots, r}$ only:
\begin{align}\label{eq:sminors1}
\wt\xi_{i}\Delta_{\omega_i, \omega_i}(\mathscr{G}(z))\Delta_{w_i\omega_i, \omega_i}(\mathscr{G}(qz))-&\xi_{i}\Delta_{w_i\omega_i, \omega_i}(\mathscr{G}(z))\Delta_{\omega_i, \omega_i}(\mathscr{G}(qz))\cr
&=\Lambda_i(z) \prod_{j<{i=i_{l}}}\Delta_{\omega_j, \omega_j}^{-a_{ji}}(\mathscr{G}(qz))
\prod_{j>{i=i_{l}}}\Delta_{\omega_j, \omega_j}^{-a_{ji}}(\mathscr{G}(z))\,, 
\end{align}
where $\xi_i$ and $\wt \xi_i$ are given in \eqref{xi} and the ordering is inherited from the inverse Coxeter element $c^{-1}$.
\end{Prop}

\begin{proof}
With the help of \eqref{eq:SinvLiftCalculation} the right hand side of \eqref{eq:Zgv} reads 
$$
\mathscr{G}(z)c^{-1}(z)v^+_{\omega_i} = L_i(z) \, \mathscr{G}(z) \tilde{s}^{-1}(z)v^+_{\omega_i}= L_i(z) \mathscr{G}(z) v^+_{c^{-1}\cdot\omega_i}\,,
$$
where $L_i(z)=\prod_i \Lambda_i^{d_i}(z)$. From \eqref{eq:Zgv} we get
\begin{equation}
Z^{-1}\mathscr{G}(qz)v^+_{\omega_i} =  L_i(z) \mathscr{G}(z) v^+_{c^{-1}\omega_i}\,.
\end{equation}
From this relation we can deduce how q-shifted generalized minors are related to unshifted ones:
\begin{equation}
\Delta_{w\cdot\omega_i,c^{-1}\omega_i} (\mathscr{G}(z)) = \Bigg[\prod_j\zeta^{\langle\check{\alpha}_j, w\cdot \omega_i\rangle}_{j}\Bigg]\cdot \frac{\Delta_{\omega_i,\omega_i}(\mathscr{G}(qz))}{L_i(z)}
\end{equation}
\begin{equation}
\Delta_{\omega_i,c^{-1}\omega_i} (\mathscr{G}(z)) = \zeta_i\cdot \frac{\Delta_{\omega_i,\omega_i}(\mathscr{G}(qz))}{L_i(z)}\,,\qquad \Delta_{w_i\cdot\omega_i,s^{-1}\omega_i} (\mathscr{G}(z)) = \zeta_i^{-1} \prod_{j\neq i}\zeta_j^{-a_{ji}}\cdot \frac{\Delta_{w_i\cdot\omega_i,\omega_i}(\mathscr{G}(qz))}{L_i(z)}\,,
\end{equation}
and so on. Thus the quadratic relation reads
\begin{align}\label{eq:sminors2}
\left(\zeta_i^{-1}\prod_{j\neq i}\zeta_j^{-a_{ji}}\right) &\Delta_{\omega_i, \omega_i}(\mathscr{G}(z))\Delta_{w_i\cdot\omega_i, \omega_i}(\mathscr{G}(qz))-\zeta_{i}\,\Delta_{w_i\omega_i, \omega_i}(\mathscr{G}(z))\Delta_{\omega_i, \omega_i}(\mathscr{G}(qz))\cr
&=\left(\prod_{j<i}\zeta_j^{-a_{ji}}\right) \frac{L_i(z)}{L_{i-1}^{-a_{ji}}(z)} \prod_{j<{i=i_{l+1}}}\Delta_{\omega_j, \omega_j}^{-a_{ji}}(\mathscr{G}(qz))
\prod_{j>{i=i_{l+1}}}\Delta_{\omega_j, \omega_j}^{-a_{ji}}(\mathscr{G}(z)), 
\end{align}
which after dividing by common factors yields
\begin{align}\label{eq:sminors3}
\left(\zeta_i^{-1}\prod_{j>i}\zeta_j^{-a_{ji}}\right) &\Delta_{\omega_i, \omega_i}(\mathscr{G}(z))\Delta_{w_i\cdot\omega_i, \omega_i}(\mathscr{G}(qz))-\left(\zeta_{i}\prod_{j<i}\zeta_j^{a_{ji}}\right)\Delta_{w_i\cdot \omega_i, \omega_i}(\mathscr{G}(z))\Delta_{\omega_i, \omega_i}(\mathscr{G}(qz))\cr
&=\frac{L_i(z)}{L_{i-1}^{-a_{ji}}(z)} \prod_{j<{i=i_{l+1}}}\Delta_{\omega_j, \omega_j}^{-a_{ji}}(\mathscr{G}(qz))
\prod_{j>{i=i_{l+1}}}\Delta_{\omega_j, \omega_j}^{-a_{ji}}(\mathscr{G}(z)), 
\end{align}

Finally, using \eqref{eq:DefLF} we can demonstrate by explicit calculation that
\begin{equation}
\frac{L_i(z)}{L_{i-1}^{-a_{ji}}(z)}  = \Lambda_i(z)\,.
\end{equation}
\end{proof}

Thus, identifying 
\begin{eqnarray}\label{identif}
\Delta_{\omega_i, \omega_i}(\mathscr{G}(z))\to Q_+^i(z), \quad \Delta_{w_i\omega_i, \omega_i}(\mathscr{G}(z))\to Q^i_-(z),  
\end{eqnarray}
we obtain that the familiar nondegenerate $QQ$ system  \eqref{qq} is equivalent to \eqref{eq:sminors1}.

Moreover, the following Theorem holds.

\begin{Thm}\label{1to1wrop}
1) Let $(\mathcal{F}_G, \mathcal{F}_{B+}, \mathscr{G}, Z)$ be a non-degenerate  $(G,q)$-Wronskian with regular singularities parametrized by the polynomials $\{\Lambda_i(z)\}_{i=1,\dots, r}$. The lower-triangular part $v(z)\in B_-(z)$ of the Gaussian decomposition $\mathscr{G}=v(z)u(z)$, $u(z)\in N_+(z)$ defines a nondegenerate $Z$-twisted Miura $(G,q)$-oper connection with regular singularities by the formula $A(z)=v^{-1}(qz)Zv(z)$.

2) There is a one-to-one correspondence between classes of nondegenerate $(G,q)$-Wronskians with regular singularities as stated in the Proposition \ref{nonuniqdef} and nondegenerate $Z$-twisted Miura $(G,q)$-opers with regular singularities parametrized by the same $\{\Lambda_i\}_{i=1,\dots,r}$, such that zeroes of the polynomials in the extended $QQ$-system are $q$-distinct from $\{\Lambda_i\}_{i=1,\dots,r}$. 
\end{Thm}

\begin{proof}
$(\mathcal{F}_G, \mathcal{F}_{B+}, \mathscr{G}, Z)$ be a non-degenerate $(G,q)$-Wronskian with regular singularities parametrized by the polynomials $\{\Lambda_i(z)\}_{i=1,\dots, r}$. 
Let us apply the relation (\ref{eq:sminors1}) from the Proposition \ref{propfundrel} to $\tilde{w}\cdot\mathscr{G}(z)$ for all $w\in W$. By Proposition \ref{weylwr} we know that $\tilde{w}\cdot\mathscr{G}(z)$ is a 
$(G,q)$-Wronskian $(\mathcal{F}_G, \mathcal{F}_{B+}, \mathscr{G}, w(Z))$. Thus Proposition \ref{propfundrel} implies that generalized minors $\Delta_{w\cdot\omega_i,\omega_i}$ generate the full $QQ$-system through the generalization of identification (\ref{identif}):
\begin{eqnarray}
\Delta_{w\cdot \omega_i, \omega_i}(\mathscr{G}(z))=\Delta_{ \omega_i, \omega_i}(\tilde{w}^{-1}\mathscr{G}(z))\to Q_+^{w,i}(z).
\end{eqnarray}
The resulting minors $\Delta_{w\cdot\omega_i,\omega_i}(\mathscr{G}(z))$ determine $v(z)$ entirely and thus produce an element which defines $Z$-twisted Miura $(G,q)$-oper as stated in the theorem. That proves part 1). To prove part 2) let us construct $\mathscr{G}(z)$ explicitly given $Z$-twisted Miura $(G,q)$-oper, so that its $q$ -connection is given by the formula
\begin{equation}\label{triv}
A(z)=v^{-1}(qz)Zv(z),
\end{equation}
where $v(z)\in B_-(z)$. 
Note, that 
\begin{eqnarray}\label{invopcond}
A^{-1}(z)=n_+(z)s_{\Lambda}^{-1}(z)\tilde{n}_+(z), \quad  n_+(z), \tilde{n}_+(z)\in N_+(z).
\end{eqnarray}
Thus, combining  (\ref{triv}), (\ref{invopcond}) we obtain
\begin{eqnarray}
Z^{-1}v(qz)=v(z)n_+(z)s_{\Lambda}^{-1}(z)\tilde{n}_+(z)
\end{eqnarray}
and 
\begin{equation}\label{formulaforg}
\mathscr{G}(z)=v(z)n_+(z)
\end{equation}
satisfies the familiar equation
\begin{eqnarray}\label{eq:Zgv}
Z^{-1}\mathscr{G}(qz)\nu^+_{\omega_i}=\mathscr{G}(z)s_{\Lambda}^{-1}(z)\nu^+_{\omega_i}.
\end{eqnarray}
Notice, that the constructed $\mathscr{G}(z)$ is defined modulo the transformations from the Proposition \ref{nonuniqdef}. This is related to the fact that the choice of $n_+(z)$ in the gauge class of $A^{-1}(z)=n_+(z)s_{\Lambda}^{-1}(z)\tilde{n}_+(z)$ is non-unique but again is up to the multiplication on the elements from Proposition \ref{nonuniqdef}. This proves the second part of the Theorem.
\end{proof}

In the next section, we will introduce the unique element in the family of $(G,q)$-Wronskians corresponding to a given Miura $(G,q)$-oper, which is a generalization of a standard $q$-Wronskian considered in \cite{KSZ} for any simply-connected simple group $G$.

\subsection{Universal quantum Wronskian for Miura $(G,q)$-oper}

In this section, we assume that the Lie group $G$ has an even Coxeter number $h$ and a choice of a Coxeter element is such that $c^{h/2}=w_0$. That only excludes $SL(N)$ case for $N$ odd, which was studied in detail in \cite{KSZ, Koroteev:2020tv}. 

The  $Z$-twisted condition for $(G,q)$-oper, which was instrumental in our considerations can be restated in the following way:
$$
Z^{-1}g(qz)=g(z)A^{-1}(z).
$$ 
One could iterate this relation to introduce a collection of relations of the following form:
\begin{eqnarray}\label{iter}
Z^{-k}g(q^{k+1}z)=g(z)A^{-1}(z)A^{-1}(qz)\dots A^{-1}(q^kz)
\end{eqnarray}

The following Lemma is true and a direct consequence of the property of the multiplication of Bruhat cells.

\begin{Lem}\label{Th:LemmaAAA}
The product $$A^{-1}(z)A^{-1}(qz)\dots A^{-1}(q^kz)$$ belongs to the Bruhat cell $B^+(z)s^{-k}B^+(z)$ as long as $0\le k<h/2$, where $h$ is the Coxeter number of G.
\end{Lem}

Let us use now the system of equations (\ref{iter}) to construct a universal $(G,q)$-Wronskian element associated with a given $Z$-twisted Miura $(G,q)$-oper in a similar way we did with the first of them in the proof of Theorem \ref{1to1wrop}. Namely, the following Theorem is true.
  
\begin{Prop}\label{Th:DifferenceEqWr}
For a given $Z$-twisted $(G,q)$-Miura oper, there exists a unique  $(G,q)$-Wronskian 
$$\mathscr{W}(z)\in B_-(z)w_0B_-(z)\cap B_+(z)w_0B_+(z)\subset G(z),$$ satisfying the system of equations 
\begin{eqnarray}\label{eq:ZWeq}
&&\mathscr{W}(q^{k+1}z)\nu^+_{\omega_i}=Z^{k}\mathscr{W}(z)s^{-1}(z)s^{-1}(qz)\dots s^{-1}(q^kz)\nu^+_{\omega_i}\,,\nonumber\\
&& i=1,\dots, r, \qquad k=0, 1, \dots, h/2-1,
\end{eqnarray}
where $h$ is the Coxeter number of $G$.
\end{Prop}
\begin{proof}
Let us use gauge transformations to reduce $A^{-1}(z)$ to the following form: 
$$A^{-1}(z)=n^1_+(z)s^{-1}(z),$$ 
where  $s(z)n^1_+(z)s^{-1}(z)\in N_-(z)$ by applying the version of Theorem \ref{canform} to $A^{-1}(z)$. We remind, that it is a unique element in the $N_+(z)$-gauge class of $(G,q)$-opers. Therefore, the element $\mathscr{W}^1(z)=g(z)n^1_+(z)$ satisfies \eqref{eq:ZWeq}. 

Now let us have a look at the 
product $A^{-1}(z)A^{-1}(qz)=n^1_+(z)s^{-1}(z)n^1_+(qz)s^{-1}(qz)$. This is an element from $N_+(z)s^{-1}(z)s^{-1}(qz)N_+(z)$ and thus can be written as $A^{-1}(z)A^{-1}(qz)=n^1_+(z)n^2_+(z)s^{-1}(z)s^{-1}(qz)\tilde{n}^2_+(z)$, so that
$$
s(qz)s(z)n^2_+(z)s^{-1}(z)s^{-1}(qz)\in N_-(z)\,,
$$
for some $n^2_+(z),\, \tilde{n}_+^2(z)\in N_+(z)$, and 
\begin{eqnarray}
s^{-1}(z)n^1_+(qz)s^{-1}(qz)=n^2_+(z)s^{-1}(z)s^{-1}(qz)\tilde{n}^2_+(z).
\end{eqnarray}
Multiplying by 
$s(qz)s(z)$ on both sides, we obtain:
\begin{eqnarray}
s(qz)n^1_+(qz)s^{-1}(qz)=s(qz)s(z)n^2_+(z)s^{-1}(z)s^{-1}(qz)\tilde{n}^2_+(z),
\end{eqnarray}
so that $\tilde{n}^2_+(z)=1$ and $n^1_+(qz)=s(z)n^2_+(z)s^{-1}(z)$. Thus we obtain that
$\mathscr{W}^2(z)=g(z)n^1_+(z)n^2_+(z)$ satisfies the second equation from  \eqref{eq:ZWeq}.

We can now proceed with the inductive step. Assume $\mathscr{W}^{k-1}(z)=g(z)n^1_+(z)\dots n_+^{k-1}(z)$ satisfies \eqref{eq:ZWeq}. 
According to Lemma \ref{Th:LemmaAAA}
$$
A^{-1}(z)A^{-1}(qz)\cdots A^{-1}(q^{k-2}z)\in N_+(z) s^{-1}(z)s^{-1}(qz)\cdots s^{-1}(q^{k-2}z) N_+(z)\,.
$$ 
Consider 
$$\prod_{l=0}^{k-1}A(q^l z)=n_+^1(z) n_+^2(z) \cdots n_+^{k-1}(z) \cdot s^{-1}(z)s^{-1}(qz)\cdots s^{-1}(q^{k-2}z) \cdot n_+^1(q^{k-1}z) s^{-1}(q^{k-1}z)\,.
$$
This product should belong to $N_+ \prod_{l=0}^{k-1} s^{-1}(q^l z) N_+$ so we can rewrite it as
$$
\prod_{l=1}^{k-1} n^l_+(z) \cdot n^k_+(z) \prod_{l=0}^{k-1} s^{-1}(q^l z) \cdot \tilde n_+^k(z)\,,
$$
provided that $\prod_{l=k-1}^{0} s(q^l z) \cdot n^k_+(z) \cdot \prod_{l=0}^{k-1} s^{-1}(q^l z)\in N_-(z)$ and
$$
\prod_{l=0}^{k-2} s^{-1}(q^l z) \cdot n_+^1(q^{k-1}z) s^{-1}(q^{k-1}z) = n^k_+(z) \cdot \prod_{l=0}^{k-1} s^{-1}(q^l z)\,  \tilde n_+^k(z)\,.
$$
Multiplying both parts by $\prod_{l=k-1}^{0} s(q^l z)$ we get 
$$
s(q^{k-1}z)n^1_+(q^{k-1}z)s^{-1}(q^{k-1}z) = \prod_{l=k-1}^{0} s(q^l z) \cdot n^k_+(z) \cdot \prod_{l=0}^{k-1} s^{-1}(q^l z)\,  \tilde n_+^k(z)\,,
$$
which implies that $\tilde n_+^k(z)=1$ and
\begin{equation}\label{eq:n1nkrelation}
n^1_+(q^{k-1}z) = \prod_{l=k-2}^{0} s(q^l z) \cdot n^k_+(z) \prod_{l=0}^{k-2} s^{-1}(q^l z)\,,
\end{equation}
which completes the inductive step so that $\mathscr{W}^{k}(z)=g(z)n^1_+(z)\dots n_+^{k}(z)$ satisfies \eqref{eq:ZWeq}.

In particular, the element 
$$\mathscr{W}^{h/2-1}(z)=g(z)n^1_+(z)n^2_+(z)\dots n_+^{h/2-1}(z),$$
satisfies all $h/2$ equations \eqref{eq:ZWeq}. Its uniqueness follows from the uniqueness of the decomposition of the $q$-oper $A(z)$.

\end{proof}

\begin{Rem}

Note that we use the relation \eqref{eq:n1nkrelation} to get 
$$n^k_+(z)=\prod_{l=0}^{k-2} s^{-1}(q^l z) n^1_+(q^{k-1}z)\prod_{l=k-2}^{0} s(q^l z)$$ 
for $k\geq 2$ and express the $(G,q)$-Wronskian in terms of $g(z), n^1_+(z)$ and $s(z)$ only:
\begin{equation}
\mathscr{W}^{k}(z)=g(z)\, n^1_+(z) \left(s^{-1}(z) n^1_+(q z) s(z) \right) \cdots  \left(\prod_{l=0}^{k-2} s^{-1}(q^l z) n^1_+(q^{k-1}z)\prod_{l=k-2}^{0} s(q^l z)\right)\,.
\end{equation}
\end{Rem}

\subsection{Example.}\label{Sec:ExplicitExamples}
Let $G=SL(r+1)$. The quantum Wronskian consists of $r+1$ columns
\begin{equation}
\mathscr{W}(z) =\left( \Delta_{\textbf{w} \omega,\omega} \Big\vert  \Delta_{\textbf{w}\omega,s^{-1}\omega} \Big\vert \dots \Big\vert  \Delta_{\textbf{w}\omega,s^{r+1}\omega}\right)(\mathscr{W}(z))\,,
\end{equation}
where 
We have $\nu^+_\omega = (1\, 0\,\dots\,0)$, $Z = \text{diag}(\xi_1,\dots,\xi_{n})$, where $\xi_1=\zeta_1, \xi_i = \zeta_{i}/\zeta_{i-1}$ for $i=2,\dots, r$ and $\xi_{r+1}=1/\zeta_r$.
According to \eqref{eq:SinvLiftCalculation} if we pick  standard ordering along the Dynkin diagram we have
$$
s^{-1}_{\Lambda}(z)= \tilde s^{-1} \prod_i \Lambda_i^{d_i}\,,
$$
where $d_i = \sum_{j=1}^i \check{\alpha}_j$ and
$$
\tilde s^{-1} = \begin{pmatrix}
0 & 0 & \dots & 0 & 1\\
1 & 0 & \dots & 0 & 0 \\
0 & 1 &\dots & 0 & 0 \\
\vdots & \vdots & \ddots & \cdots & \vdots\\
0 & 0 & \dots & 1 & 0
\end{pmatrix}\,,
$$
so that $\tilde s^{-l} \nu^+_\omega = (0,\dots, 0, 1\, 0\,\dots\,0)$, where $1$ is on the $l$th place.
Thus, according to Proposition \ref{propfundrel} the $q$-Wronskian reads
\begin{equation}\label{WexplSLn}
\mathscr{W}(z) =\left(Q^{\textbf{w}\cdot\omega}(z) \Big\vert  Z F_1(z)Q^{\textbf{w}\cdot\omega}(qz)\Big\vert \dots \Big\vert  Z^{r-1}F_{r-1}(q^{r-1}z)Q^{\textbf{w}\cdot\omega}(q^{r-1}z)\right)\,,
\end{equation}
where $F_i(z)=\prod_{j=1}^i \Lambda_j(z)^{-1}$. 

The conditions for the dual $(SL(r+1),q)$-oper, according to Theorem \ref{1to1wrop} can be formulated using the above matrix and they were first formulated in \cite{KSZ}, (see equation (4.8) albeit written in a slightly different convention and normalization). The condition corresponding to the whole $q$-Wronskian reads
$\text{det} \mathscr{W}(z)=1$, whereas the others can be readily written using minors of matrix $\mathscr{W}(z)$.

For the type A root system the relation \eqref{eq:minorsgen} reads
\begin{equation}\label{eq:minorssln}
\Delta_{u\omega_i, v\omega_i}\Delta_{us_i\omega_i, vs_i\omega_i}-
\Delta_{us_i\omega_i, v\omega_i}\Delta_{u\omega_i, vs_i\omega_i}=\Delta_{u\omega_{i-1}, v\omega_{i-1}}\Delta_{u\omega_{i+1}, v\omega_{i+1}}\,,
\end{equation}
which as we have shown previously are equivalent to the corresponding $QQ$-system. As was discussed in \cite{Koroteev:2017aa, KSZ} these equations can be reduced to the following determinant identity known from the 19th century (Desnanot-Jacobi-Lewis Carroll Identity) using matrix of the form \eqref{WexplSLn}. 
\begin{equation}\label{eq:LCarrol12i}
M^1_1 M^2_i- M^1_i M^2_1= M^{12}_{1i} M\,,
\end{equation}
where $M^a_b$ is the determinant of the quantum Wronskian matrix $\mathscr{W}(z)$ with the $a$th row and $b$th column removed and $M=\text{det}\mathscr{W}(z)$.

The identification between \eqref{eq:minorssln} and \eqref{eq:LCarrol12i} works as follows. We put $u=1$ and $v=s_1\cdot s_2\cdots s_{i-1}$. This way $vs_i = s_1\cdots s_{i}$ is the element that permutes the first the last column of matrix $M$ as well as
\begin{equation}
M = \Delta_{\omega_{i+1}, v\omega_{i+1}}\,,\quad M^1_1 = \Delta_{\omega_i, v\omega_i}\,,\quad M^2_i = \Delta_{s_i\omega_i, vs_i\omega_i}\,,\quad M^2_1 = \Delta_{s_i\omega_i, v\omega_i}\,,\quad M^1_i = \Delta_{\omega_i, vs_i\omega_i}
\end{equation}

In other words, after acting with element $v$ on the columns the Lewis Carroll identity can be presented in terms of principal minors
\begin{equation}
{\widetilde M}^1_1 {\widetilde M}^2_2- {\widetilde M}^1_2 {\widetilde M}^2_1= {\widetilde M}^{12}_{12} {\widetilde M}\,,
\end{equation}
where $\widetilde{M}=M\cdot v$.

\bibliography{cpn1}
\end{document}